\documentclass[11pt]{amsart}
\pdfoutput=1
\usepackage[margin=3.5cm]{geometry}

\usepackage[T1]{fontenc}
\usepackage{fourier}
\usepackage{microtype}

\usepackage{url} 
\usepackage{caption}
\usepackage{verbatim} 
\usepackage{graphicx}
\usepackage{pinlabel}
\usepackage[pdftex, curve, graph]{xy}

\newcommand{\bR}{\mathbb R}
\newcommand{\bN}{\mathbb N}
\newcommand{\C}{\mathcal C}
\newcommand{\R}{\mathbb R}
\newcommand{\V}{\mathcal V}
\newcommand{\W}{\mathcal W}
\newcommand{\E}{\mathcal E}
\newcommand{\bZ}{\mathbb Z} 
\newcommand{\Z}{\mathcal Z}
\newcommand{\F}{\mathcal F}

\newcommand{\rk}{{\rm rk} \,}
\newcommand{\Mod}{\rm Mod}

\newcommand{\al}{\alpha}
\newcommand{\be}{\beta}
\newcommand{\de}{\delta}

\newcommand{\si}{\sigma}
\newcommand{\ga}{\gamma}
\newcommand{\De}{\Delta}
\newcommand{\Ga}{\Gamma}
\newcommand{\La}{\Lambda}

\newcommand{\del}{\partial}
\newcommand{\co}{\thinspace\colon}

\newcommand{\Tor}{\rm Tor}
\DeclareMathOperator{\Diff}{Diff}
\newcommand{\lk}{e}

\newtheorem{thm}{Theorem}[section]
\newtheorem{lemma}[thm]{Lemma}
\newtheorem{cor}[thm]{Corollary}
\newtheorem{rem}[thm]{Remark}
\newtheorem{prop}[thm]{Proposition}

\theoremstyle{definition}
\newtheorem{defn}[thm]{Definition}

\newtheorem{rmk}[thm]{Remark}
\newtheorem*{rmks}{Remarks}
\newtheorem*{exas}{Examples}
\newtheorem*{problem}{Problem}

\numberwithin{equation}{section}

\begin{document}

\title[On 3--braid knots of finite concordance order]{On 3--braid knots of finite concordance order}

\author{Paolo Lisca}
\address{Dipartimento di Matematica\\ 
Largo Bruno Pontecorvo, 5\\
Universit\`a di Pisa \\
I-56127 Pisa, ITALY} 
\email{lisca@dm.unipi.it}
\thanks{The author was partially supported by the PRIN--MIUR research project 2010--2011 
``Variet\`a reali e complesse: geometria, 
topologia e analisi armonica''.}

\date{}

\begin{abstract} 
We study 3--braid knots of finite smooth concordance order. A corollary of our main result is that a 
chiral 3--braid knot of finite concordance order is ribbon. 
\end{abstract}

\maketitle

\section{Introduction}\label{s:intro}

In this paper we investigate 3--braid knots of finite concordance order. We work 
in the smooth category, therefore the words `slice' and `concordance' will always mean,  
respectively, `smoothly slice' and `smooth concordance'. The {\em reverse} of an oriented knot will be 
denoted $-K$, and $K^m$ will denote the {\em mirror image} of $K$. In this notation, a connected sum of the 
form $K\#(-K^m)$ is always a {\em slice} knot, i.e.~it bounds a properly embedded disk $D\subset B^4$. 
Recall that a knot is {\em ribbon} if it bounds a properly embedded disk $D\subset B^4$ such that the 
restriction of the radial function $B^4\to [0,1]$ has no local maxima in the interior of $D$. That each 
slice knot is ribbon is the content of the well--known slice--ribbon conjecture. 
A knot $K\subset S^3$ is {\em amphichiral} if $K^m$ is isotopic to 
either $K$ or $-K$. The knot $K$ is {\em chiral} if it is not amphichiral.
 Recall that the classical concordance group $\C$ is the set of equivalence 
classes $[K]$ of oriented knots $K\subset S^3$ with respect to the equivalence relation which declares  
$K_1$ and $K_2$ equivalent if $K_1\#(-K_2^m)$ is slice. The group operation is induced by 
connected sum and $0\in\C$ is the concordance class of the unknot. Several facts are known about 
the structure of $\C$, but its torsion subgroup is not understood, see e.g.~\cite{Li05}. 
We will say that a knot $K$ is of {\em finite concordance order} if $[K]\in\C$ is of finite order.

In~\cite{Ba08} Baldwin obtained some information on 3--braid knots of finite concordance order by 
computing a certain correction term of the Heegaard Floer homology of the two--fold branched cover 
(see Section~\ref{s:prelim} for more details). We will use his result together with constraints obtained via 
Donaldson's `Theorem~A'~\cite{Do87} to establish Theorem~\ref{t:main} below, 
which is our main result. Our approach here is similar to the one used in~\cite{Li07-1,Li07-2}, 
where the concordance orders of 2--bridge knots are determined. 
We point out that the idea of combining information coming 
from Heegaard Floer correction terms with Donaldson's Theorem~A 
was also used in~\cite{Do15, GJ11,Gr14, Le12, Le13}.

Before we can state Theorem~\ref{t:main} we need to introduce some terminology. 
Recall that a {\em symmetric union} knot is a special kind of ribbon knot first introduced by 
Kinoshita and Terasaka~\cite{KT57}. It is unknown whether every ribbon knot is a symmetric union. 
A braid $\be\in B_n$ is called {\em quasi--positive} if it can written as a product of conjugates of 
the standard generators $\si_1,\ldots, \si_{n-1}\in B_n$, and {\em quasi--negative} if $\be^{-1}$ is quasi--positive. 
A knot $K$ is {\em quasi--positive} (respectively {\em quasi--negative}) if $K$ is the closure of a 
quasi--positive (respectively quasi--negative) braid. 

We now recall the notion of `blowup' from~\cite{Li14}. Let $\bN$ be the set of (positive) 
natural numbers, $\bN_0 := \bN\cup\{0\}$, and let $k\in\bN$.
We say that $\hat z\in\bN_0^{k+1}$ is a  {\em blow--up} of $z=(n_1,\ldots,n_k)\in\bN_0^k$ if  
\[
\hat z = 
\begin{cases} 
(1,n_1+1,n_2,\ldots,n_{k-1},n_k+1),\ \text{or} \\
(n_1,\ldots,n_i+1,1,n_{i+1}+1,\ldots,n_k),\ \text{for some}\ 1\leq i < k,\ \text{or} \\
(n_1+1,n_2,\ldots,n_{k-1},n_k+1,1).  
\end{cases} 
\]
There is a well--known isomorphism between the 3--braid group and the mapping 
class group of the one--holed torus. Let $T$ be an oriented, one--holed torus, and let $\Mod(T)$ be the 
group of isotopy classes of orientation--preserving diffeomorphisms of $T$, where isotopies are 
required to fix the boundary pointwise. Let $x, y \in\Mod(T)$ be right--handed Dehn twists along two 
simple closed curves in $T$ intersecting transversely once. The group $\Mod(T)$ is generated by $x$ 
and $y$ subject to the relation $xyx=yxy$, and there is an isomorphism from $\psi\co B_3\to\Mod(T)$ 
sending the standard generators $\si_1, \si_2$ to $x$, respectively $y$. The isomorphism $\psi$ 
can be realized geometrically by viewing $T$ as a two--fold branched cover over the $2$--disk 
with three branch points: elements of $B_3$, viewed as automorphisms of the triply--pointed disk, 
lift uniquely to elements of $\Mod(T)$. It is easy to check that an element $h\in\Mod(T)$ is the image 
under $\psi$ of a quasi--positive 3--braid if and only if $h$ can be written as a product of right--handed 
Dehn twists.

Keeping the isomorphism $\psi$ in mind, it is easy to check that by~\cite[Theorem~2.3]{Li14}, 
if $(s_1,\ldots, s_N)$ is obtained from $(0,0)$ via a sequence of blowups and the string 
\[
(c_1,\ldots, c_N) := (x_1+2,\overbrace{2,\ldots,2}^{y_1-1},\ldots,x_t+2,\overbrace{2,\ldots,2}^{y_t-1})
\]
satisfies $c_i\geq s_i$ for $i=1,\ldots, N$, then the 3--braid $(\si_1\si_2)^3\prod_{i=1}^t \si_1^{x_i}\si_2^{-y_i}$ 
is quasi--positive. 

Observe that any string obtained from $(0,0)$ via a sequence of blowups contains always at least two $1$'s and, 
typically, more than two $1$'s. 
\begin{exas}\label{ex:qp}
(1) According to Knotinfo~\cite{CL15}, the knot $12_{n0721}$ is slice, chiral, and equal 
to the closure of  the 3--braid $\al=\si_1\si_2^2\si_1^4\si_2^{-5}$. Applying~\cite[Proposition~2.1]{Mu74} it is easy to check that 
$\al$ is conjugate to $(\si_1\si_2)^3\si_1^3\si_2^{-7}$, 
whose associated string of integers is obtained by changing the two $1$'s into $2$'s in $(5,1,2,2,2,2,1)$. 
Moreover, $(5,1,2,2,2,2,1)$ is an iterated blowup of $(0,0)$: 
\[
(5,1,2,2,2,2,1)\rightarrow (4,1,2,2,2,1)\rightarrow (3,1,2,2,1)\rightarrow (2,1,2,1)\rightarrow (1,1,1)\rightarrow (0,0).
\]
It follows from~\cite[Theorem~2.3]{Li14} that $\al$ is quasi--positive. 

(2) As another example one may consider the knot $12_{n0708}$, which according to Knotinfo is slice, 
chiral and equal to the closure of the 3--braid $\be=\si_1\si_2^{-3}\si_1\si_2^{-1}\si_1\si_2\si_1^{-1}\si_2^3$. 
Applying~\cite[Proposition~2.1]{Mu74} one can check that $\be$ is conjugate to 
$(\si_1\si_2)^3\si_1\si_2^{-3}\si_1^2\si_2^{-4}$, corresponding to the string obtained 
changing the two $1$'s into $2$'s in $(3,1,2,4,1,2,2)$, which is an iterated blowup of $(0,0)$. 
Therefore $12_{n0708}$ is also quasi--positive.
Previously, the quasi--positivity of $12_{n0708}$ and $12_{n0721}$ appear to have been unknown; 
compare~\cite{CL15}~\footnote{Although there exists an algorithm to establish the quasi--positivity 
of any 3--braid (but not the quasi--positivity of any 3--braid closure)~\cite{Or04}.}. 
\end{exas}

Before we can state our main result we need a little more terminology. Let $P\subset\R^2$ be 
a regular polygon with $t\geq 2$ vertices. Let $V_P=\{v_1,\ldots, v_t\}$ and $E_P=\{e_1,\ldots, e_t\}$ be 
the sets of vertices and edges of $P$, indexed so that, for each $i=1,\ldots, t-1$, the edge $e_i$ is between 
$v_i$ and $v_{i+1}$. A~{\em labelling} of $P$ will be a pair $(X,Y)$ of maps $X\co V_P\to\bN$ and $Y\co E_P\to\bN$. 
We say that a $2t$--uple $(x_1,y_1,x_2,y_2,\ldots, x_t,y_t)\in\bN^{2t}$ {\em encodes} a labelling $(X,Y)$ of $P$ 
if 
\[
(x_1,y_1,x_2,y_2,\ldots, x_t,y_t) = (X(v_1),Y(e_1),\ldots, X(v_t), Y(e_t)). 
\]
Given a labelling $(X,Y)$ and a symmetry 
$\varphi\co P\to P$ of $P$, one can define a new labelling of $P$ by setting 
$(X,Y)^{\varphi}:=(X\circ\varphi_V, Y\circ\varphi_E)$, where $\varphi_V\co V_P\to V_P$ and $\varphi_E\co E_P\to E_P$ 
are the $1-1$ maps induced by $\varphi$. We are now ready to state our main result. 

\begin{thm}\label{t:main} 
Let $K\subset S^3$ be a $3$--braid knot of finite concordance order. Then, one of the 
following holds:
\begin{enumerate} 
\item
either $K$ or $K^m$ is the closure of a 3--braid $\be$ of the form
\[
\be = (\si_1\si_2)^3\si_1^{x_1}\si_2^{-y_1}\cdots\si_1^{x_t}\si_2^{-y_t},\quad t, x_i, y_i\geq 1,
\] 
where $\sum_i y_i = \sum_i x_i +4$ and the string of positive integers 
$
(x_1+2,\overbrace{2,\ldots,2}^{y_1-1},\ldots,x_t+2,\overbrace{2,\ldots,2}^{y_t-1})
$
is obtained from an iterated blowup of $(0,0)$ by replacing two $1$'s with $2$'s. 
Moreover, $\be$ is quasi--positive and $K$ is ribbon. 
\item
$K$ is a symmetric union of the form $L_a$ given in Figure~\ref{f:symmunion}, where $a\in B_3$;
\begin{figure}[ht]
\centering
\labellist
\hair 2pt
\pinlabel \large\rotatebox[origin=c]{90}{$a$} at 265 168
\pinlabel \large\rotatebox[origin=c]{270}{$a^{-1}$} at 50 168
\endlabellist
\centering
\includegraphics[scale=0.4]{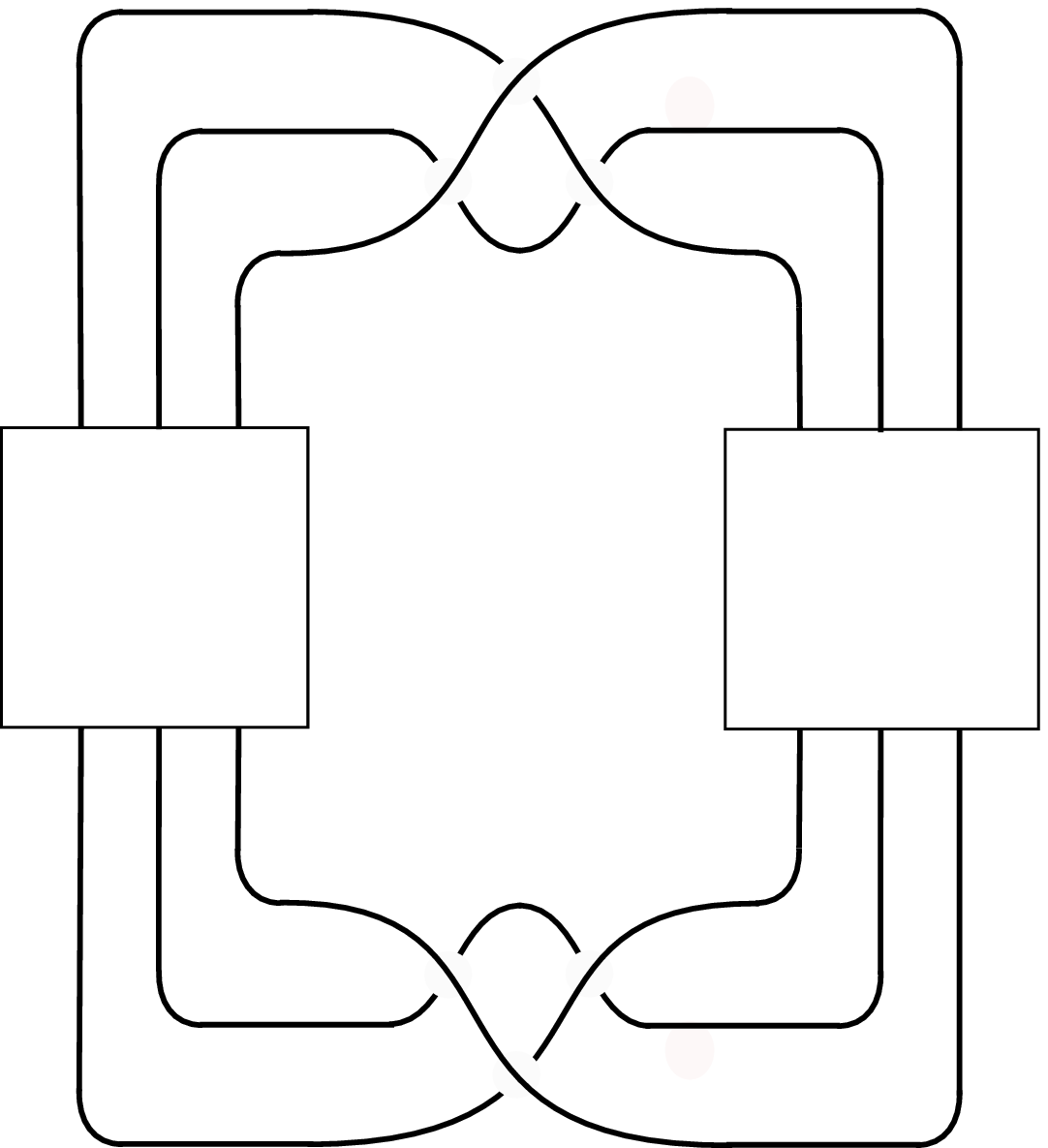}
\caption{The link $L_a$}
\label{f:symmunion} 
\end{figure} 
\item
$K$ is isotopic to the closure of a 3--braid $\be$ of the form 
\[
\be = \si_1^{x_1}\si_2^{-y_1}\cdots\si_1^{x_t}\si_2^{-y_t},\quad 
t\geq 2,\ x_1,\ldots, x_t\geq 1,
\]
where $\sum_i y_i = \sum_i x_i$ 
and there exist a regular polygon $P$ with $t$ vertices and a symmetry $\varphi\co P\to P$ such that 
$(X,Y) = (Y,X')^\varphi$, where $(X,Y)$ is the labelling encoded by $(x_1,y_1,\ldots, x_t,y_t)$ 
and $(Y,X')$ is the labelling encoded by $(y_1,x_2, y_2,x_3,\ldots, y_t,x_1)$. 
Moreover, $K$ is amphichiral.
\end{enumerate} 
\end{thm} 
It is natural to wonder how the three families of knots appearing in Theorem~\ref{t:main} intersect each other. 
It is easy to check that a knot belonging to Family~(1) is the closure of a 3--braid $\be$ 
satisfying $\lk(\be)=\pm 2$, where $\lk\co B_3\to\bZ$ denotes the abelianization homomorphism 
(exponent sum with respect to the standard generators $\si_i$). On the other hand, a knot $K$ belonging 
to Family~(2) or Family~(3) is the closure of a 3--braid $\be$ with $\lk(\be)=0$. By the main result of~\cite{BM93}, 
a link which can be represented as a 3--braid closure admits a unique conjugacy class of 3--braid representatives, 
with the exceptions of (i) the unknot, which can be represented only by the conjugacy classes of 
$\si_1\si_2$, $\si_1\si_2^{-1}$ or $\si_1^{-1}\si_2^{-1}$, (ii) a type $(2,k)$ torus link with $k\neq\pm 1$ 
and (iii) a special class of links admitting at most two conjugacy classes of 3--braid representatives 
having the same exponent sum. By the results of~\cite{Mu74} none of the 3--braids giving rise to the knots of Family~(1) is conjugate to either $\si_1\si_2$ or $\si_1^{-1}\si_2^{-1}$, therefore the unknot does not belong to the first family.  
Moreover, a $(2,k)$--torus knot with $k\neq\pm 1$ has non--vanishing signature, while clearly knots 
belonging to Family~(2) have vanishing signature. By the computations of~\cite{Er99} (see the proof of Lemma~\ref{l:linking-finite-order}) knots belonging to Family~(3) also have vanishing signature. Therefore we can 
conclude that Family~(1) is disjoint from the union of Families~(2) and~(3). 
On the other hand, there are knots belonging simultaneously to Families~(2) and~(3). 
The easiest example is the knot $8_9$, which coincides with $L_{\si_1\si_2^{-1}\si_1^2}$ and therefore belongs to 
Family~(2), while according to Knotinfo~\cite{CL15} 
it is the closure of $\si^3\si_2^{-1}\si_1\si_2^{-3}$ and therefore belongs to 
Family~(3) as well. As a final comment we point out that both Families (1) and (2) contain chiral knots. For instance, 
according to Knotinfo the 3--braid slice knot $8_{20}$ is chiral and equal to the closure of a quasi--negative 
3--braid, therefore it belongs to Family (1). 
More examples of chiral knots belonging to Family~(1) are $12_{n0708}$ and $12_{n0721}$ described in the 
examples before Theorem~\ref{t:main}. 
Similarly, the 3--braid slice knots $10_{48}$ and $12_{a1011}$ are chiral and neither quasi--positive nor 
quasi--negative, therefore they belong to Family (2). 

The following corollary follows immediately from Theorem~\ref{t:main}.

\begin{cor}\label{c:main}
A chiral 3--braid knot of finite concordance order is ribbon. 
\qed\end{cor}

\begin{rmks}
(1) Corollary~\ref{c:main} holds for 2--bridge knots. Indeed, by~\cite[Corollary~1.3]{Li07-2} 
a chiral 2--bridge knot of finite concordance order is slice, 
and by~\cite{Li07-1} a slice 2--bridge knot is ribbon. 

(2) It is not difficult to find chiral knots of finite concordance order which do not satisfy the conclusion 
of Corollary~\ref{c:main} (and therefore cannot be closures of 3--braids). For instance, according to 
Knotinfo~\cite{CL15} the chiral 4-braid knots $9_{24}$ and $9_{37}$ have concordance order $2$. 

(3) There are $185$ knots with crossing number at most twelve and braid index $3$. 
Among these, $12$ are ribbon (some of which quasi--positive or quasi--negative), $17$ are amphichiral and have 
concordance order $2$, while $151$ have infinite concordance order. The remaining knots $10_{91}$, 
$12_{a1199}$, $12_{a1222}$, $12_{a1231}$ and  $12_{a1258}$ are chiral and non--slice. 
In view of Corollary~\ref{c:main}, the five knots above have infinite concordance order. 
Previously, their concordance orders appear to have been unknown; compare~\cite{CL15}.
\end{rmks}

Theorem~\ref{t:main} naturally leads to the following problem, to which we hope to return in the future.

\begin{problem}\label{prob:family3}
Determine the concordance orders of the knots in Family (3) of Theorem~\ref{t:main}. 
\end{problem}
 
The paper is organized as follows. In Section~\ref{s:prelim} we collect some preliminary results on 
knots of finite concordance order which are 3--braid closures.  
The purpose of the section is to prove Proposition~\ref{p:embed}, which uses Donaldson's `Theorem A' to 
show that if a 3--braid knot $K$ has concordance order $k$, then the orthogonal sum of $k$ copies 
of a certain negative definite integral lattice embeds isometrically in the standard negative 
definite lattice of the same rank. In Section~\ref{s:latan} we draw the lattice--theoretical consequences 
of the existence of the isometric embedding given by Proposition~\ref{p:embed}. Section~\ref{s:latan} contains 
the bulk of the technical work. In Section~\ref{s:proof} we prove Theorem~\ref{t:main} using 
the results of Sections~\ref{s:prelim} and~\ref{s:latan}. 

\bigskip
\noindent 
{\bf Acknowledgements:} 
The author wishes to warmly thank an anonymous referee for her/his extremely thorough 
job, which helped him to correct several mistakes and to improve the exposition.

\section{Preliminaries}\label{s:prelim}
The following simple lemma uses the slice--Bennequin inequality~\cite{Ru93} to 
establish a basic property of braid closures having finite concordance order. 

\begin{lemma}\label{l:linking-finite-order}
Let $K\subset S^3$ be a knot of finite concordance order which is the closure of a braid $\be\in B_n$. 
Then, \[1-n\leq\lk(\be)\leq n-1,\] where $\lk\co B_n\to\bZ$ is the exponent sum homomorphism.
\end{lemma}

\begin{proof} 
Recall that, if a knot $N\subset S^3$ is the closure of an $n$--braid $\ga\in B_n$, 
the slice Bennequin inequality~\cite{Ru93} reads
\begin{equation}\label{e:sBI}
1 - g_s(N) \leq n - \lk(\ga),
\end{equation}
where $g_s(N)$ is the slice genus of $N$. Now observe that, for each $m\geq 1$, the knot $mK:=K\#\stackrel{(m)}{\cdots}\#K$ is the closure of the $mn$--braid $\eta$ described in Figure~\ref{f:mnbraid}. 
\begin{figure}[h] 
\labellist
\hair 2pt
\pinlabel $\be$ at 128 45
\pinlabel $\be$ at 128 260
\pinlabel $\be$ at 128 372
\pinlabel \rotatebox[origin=c]{90}{$\cdots$} at  40 47
\pinlabel \rotatebox[origin=c]{90}{$\cdots$} at  40 262
\pinlabel \rotatebox[origin=c]{90}{$\cdots$} at  40 374
\pinlabel \rotatebox[origin=c]{90}{$\cdots$} at  215 47
\pinlabel \rotatebox[origin=c]{90}{$\cdots$} at  215 262
\pinlabel \rotatebox[origin=c]{90}{$\cdots$} at  215 374
\pinlabel \rotatebox[origin=c]{90}{$\cdots$} at  310 47
\pinlabel \rotatebox[origin=c]{90}{$\cdots$} at  310 262
\pinlabel \rotatebox[origin=c]{90}{$\cdots$} at  310 374
\pinlabel \rotatebox[origin=c]{90}{$\cdots$} at  260 156
\pinlabel \rotatebox[origin=c]{90}{$\cdots$} at  128 156
\endlabellist
\centering
\includegraphics[scale=0.4]{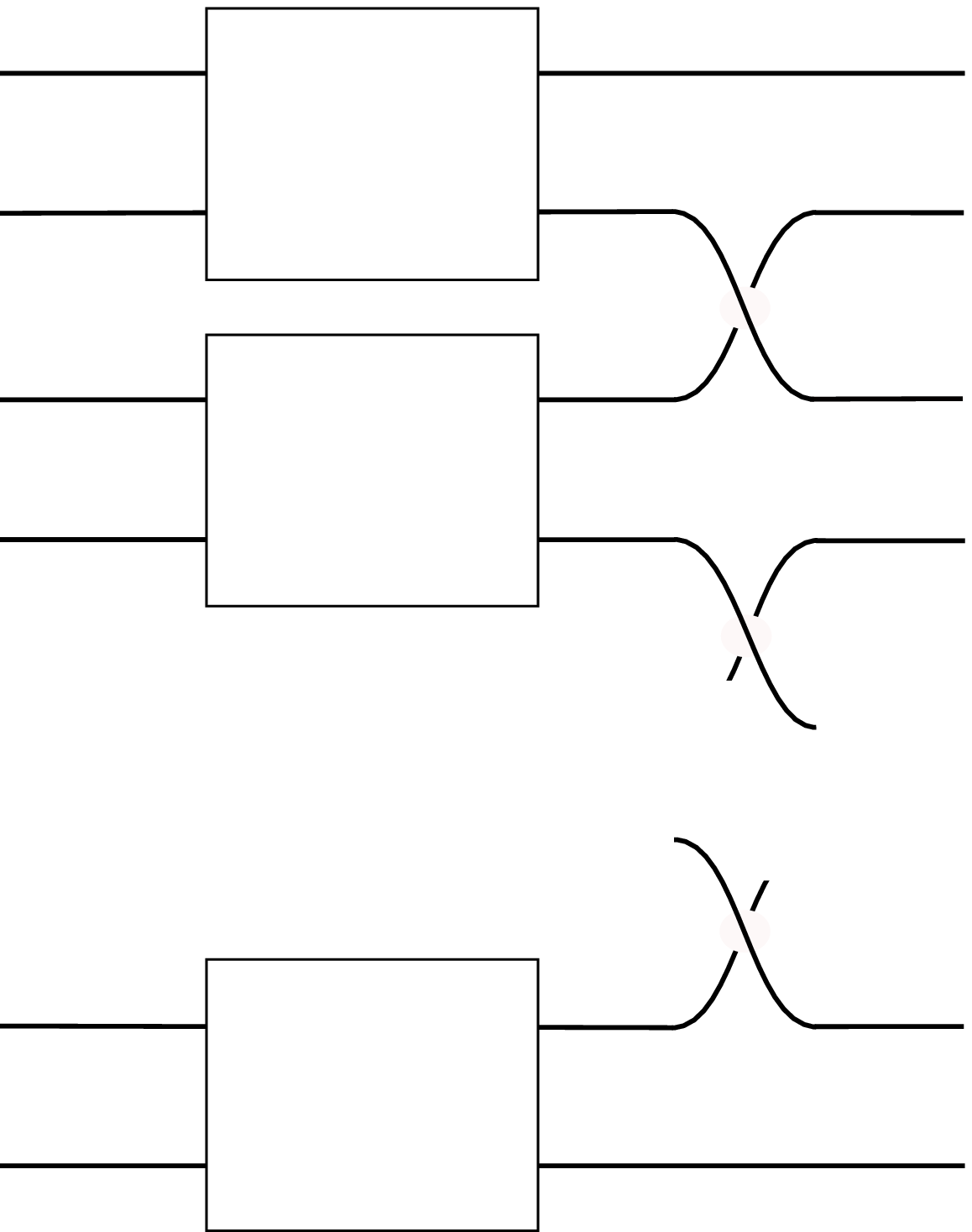}
\caption{The $mn$--braid $\eta$ with $\hat\eta=mK$}
\label{f:mnbraid}
\end{figure} 
If $K$ has concordance order $m$ then $g_s(mK)=0$. 
Since $\lk(\eta)=m\lk(\be)+m-1$, applying Inequality~\eqref{e:sBI} to $\eta$ we get 
the inequality $\lk(\be)\leq n-1$. Notice that the mirror image of $K$ is a knot of concordance order $m$ 
which is the closure of a  3--braid $\bar\be$ such that $\lk(\bar\be) = -\lk(\be)$. Arguing as before we have 
$\lk(\bar\be)\leq n-1$, and the statement follows.
\end{proof} 

In~\cite{Ba08} Baldwin combines Murasugi's normal form 
for 3--braid closures~\cite{Mu74} and the signature computations of~\cite{Er99} 
with computations of the concordance invariant $\delta$ defined by Manolescu--Owens~\cite{MO07} 
to establish the following proposition. Here we provide a short proof of Baldwin's result based 
on~\cite{Mu74, Er99} and Lemma~\ref{l:linking-finite-order}. 

\begin{prop}[{\cite[Prop.~8.6]{Ba08}}]\label{p:prelim}
Let $K\subset S^3$ be a $3$--braid knot of finite concordance order. 
Then, $K$ is the closure of a 3--braid $\be\in B_3$ of the form 
\begin{equation}\label{e:prelim}
\be = (\si_1\si_2)^{3d} \si_1^{x_1}\si_2^{-y_1}\cdots\si_1^{x_t}\si_2^{-y_t},\quad
t, x_i, y_i\geq 1, 
\end{equation}
where $d\in\{-1,0,1\}$ and $\sum_{i=1}^t (x_i-y_i) = -4d$. Moreover, if $d=\pm 1$ then  
up to replacing $K$ with $K^m$ one can take $d=1$ in Equation~\eqref{e:prelim}. 
\end{prop}

\begin{proof}
As observed in~\cite[Remark~8.4]{Ba08}, the results of~\cite{Mu74} 
immediately imply that a 3--braid knot of finite concordance order is 
either the unknot or is isotopic to the closure $\hat\be$ of a 3--braid $\be$ of the form:
\begin{equation}\label{e:proof-braid} 
\be = (\si_1\si_2\si_1)^{2d} \si_1 \si_2^{-a_1} \cdots \si_1\si_2^{-a_n},\quad a_i\geq 0,\ 
\text{some $a_j>0$},
\end{equation}
where $d\in\bZ$. By~\cite{Er99} $K$ has signature 
\begin{equation}\label{e:signature}
\si(K) = -n -4d +\sum_{i=1}^n a_i = 2d - \lk(\be),
\end{equation}
where $\lk\co B_3\to\bZ$ is the exponent sum homomorphism. Since $\si :\C\to\bZ$ is a homomorphism, 
the fact that $K$ has finite concordance order implies $\si(K)=0$, therefore $\lk(\be) = 2d$. Moreover,  
Lemma~\ref{l:linking-finite-order} implies $d\in\{-1,0,1\}$. Since 
$(\si_2\si_1\si_2)^2=(\si_1\si_2\si_1)^2=(\si_1\si_2)^3$ in $B_3$, if $d=1$ we are done. If $d=-1$ and 
$K$ is the closure of $(\si_1\si_2\si_1)^{-2} \si_1^{x_1}\si_2^{-y_1}\cdots\si_1^{x_t}\si_2^{-y_t}$ 
then $K^m$ is the closure of 
\[
(\si_1^{-1}\si_2^{-1}\si_1^{-1})^{-2} \si_1^{-x_1}\si_2^{y_1}\cdots\si_1^{-x_t}\si_2^{y_t} = 
(\si_1\si_2\si_1)^2 \si_1^{-x_1}\si_2^{y_1}\cdots\si_1^{-x_t}\si_2^{y_t}.
\]
If $f\co B_3\to B_3$ is the automorphism which sends $\si_1$ to $\si_2$ and $\si_2$ to $\si_1$, 
it is easy to check that for each $\be\in B_3$ the closure of $\be$ is isotopic to the closure of $f(\be)$.
Therefore, $K^m$ is also the closure of $(\si_2\si_1\si_2)^2 \si_2^{-x_1}\si_1^{y_1}\cdots\si_2^{-x_t}\si_1^{y_t}$, 
which is conjugate to 
\[
(\si_2\si_1\si_2)^2 \si_1^{y_1}\si_2^{-x_2}\cdots\si_1^{-x_t}\si_1^{y_t}\si_2^{-x_1}
\]
because the element $(\si_2\si_1\si_2)^2$ is central. This shows that up to replacing $K$ with $K^m$ 
we may assume $d\in\{0,1\}$, therefore, the braid $\be$ of Equation~\eqref{e:proof-braid} is of the form given in Equation~\eqref{e:prelim}. 
\end{proof}

Our next task is to show that a 2--fold cover of $S^3$ branched along a 3--braid knot of finite 
concordance order bounds a smooth 4--manifold with an intersection lattice $\La_\Ga$ associated 
with a certain weighted graph $\Ga$. 
Let $\La_\Ga$ be the free abelian group generated by the vertices of the integrally weighted graph  
$\Ga$ of Figure~\ref{f:graph}, where $d\in\bZ$, $t, x_i, y_i\geq 1$ and $\sum_i y_i\geq 2$. 
The graph $\Ga$ naturally determines a symmetric, bilinear form $\cdot:\La_\Ga\times\La_\Ga\to\bZ$ 
such that, if $v, w$ are two vertices of $\Ga$ then $v\cdot v$ equals the weight of $v$, $v\cdot w$ equals $1$ 
if $v, w$ are connected by an unlabelled edge, $(-1)^d$ if they are connected by the 
$(-1)^d$--labelled edge, and $0$ otherwise. 
\begin{figure}[h]
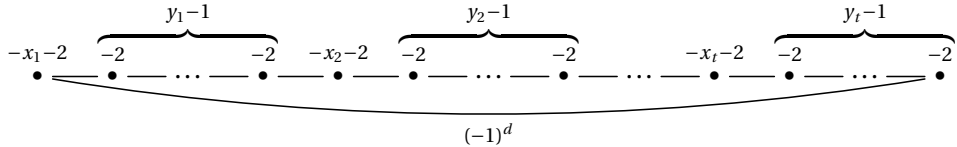

\centering
$ \xygraph{
!{<0cm,0cm>;<1cm,0cm>:<0cm,1cm>::}
!{(0,1) }*+{\bullet}="0"
!{(0,1.3)}*+{\scriptstyle -x_1-2}
!{(1,1)}*+{\bullet}="1"
!{(1,1.3)}*+{\scriptstyle -2}
!{(2,1)}*+{\ldots}="2"
!{(2,1.7)}*+{\overset{y_1 - 1}{\overbrace{\hspace{68pt}}}}
!{(3,1)}*+{\bullet}="3"
!{(3,1.3)}*+{\scriptstyle -2}
!{(4,1)}*+{\bullet}="4"
!{(4,1.3)}*+{\scriptstyle -x_2-2}
!{(5,1)}*+{\bullet}="5"
!{(5,1.3)}*+{\scriptstyle -2}
!{(6,1)}*+{\ldots}="6"
!{(6,1.7)}*+{\overset{y_2 - 1}{\overbrace{\hspace{68pt}}}}
!{(7,1)}*+{\bullet}="7"
!{(7,1.3)}*+{\scriptstyle -2}
!{(8,1) }*+{\ldots}="8"
!{(9,1) }*+{\bullet}="9"
!{(9,1.3)}*+{\scriptstyle -x_t-2}
!{(10,1) }*+{\bullet}="10"
!{(10,1.3)}*+{\scriptstyle -2}
!{(11,1) }*+{\ldots}="11"
!{(11,1.7)}*+{\overset{y_t - 1}{\overbrace{\hspace{68pt}}}}
!{(12,1) }*+{\bullet}="12"
!{(12,1.3)}*+{\scriptstyle -2}
"0"-"1"
"1"-"2"
"2"-"3"
"3"-"4"
"4"-"5"
"5"-"6"
"6"-"7"
"7"-"8"
"8"-"9"
"9"-"10"
"10"-"11"
"11"-"12"
"12"-@/^0.5cm/"0"^{(-1)^d}
}$
\caption{The weighted graph $\Ga$}
\label{f:graph}
\end{figure} 

\begin{lemma}\label{l:2foldcover}
Let $K\subset S^3$ be the closure of a 3--braid $\be\in B_3$ of the form given in Equation~\eqref{e:prelim}
for some $t, x_i, y_i\geq 1$ with $\sum_i y_i\geq 2$ and $d\in\{0,1\}$. Let $Y_K$ be the 2--fold 
cover of $S^3$ branched along $K$. Then, there is a smooth, oriented 4--manifold $M$  
whose intersection lattice $H_2(M;\bZ)/\Tor$ is isometric to $\La_\Ga$, with $\del M= - Y_K$. 
\end{lemma}

\begin{proof}
We prove the statement considering separately the two cases $d=0$ and $d=1$. 

{\em First case: $d=0$}. 
Let $\be^m$ denote the `mirror braid' obtained from $\be$ by replacing $\si_i$ with $\si_i^{-1}$, for $i=1,2$. Consider 
the regular projection $P\subset\bR^2$ of the closure $K^m = \widehat\be^m$ 
given in Figure~\ref{f:projection}.  
\begin{figure}[h]
\centering
\labellist
\hair 2pt
\pinlabel $\be^m$ at 113 78
\endlabellist
\centering
\includegraphics[scale=0.6]{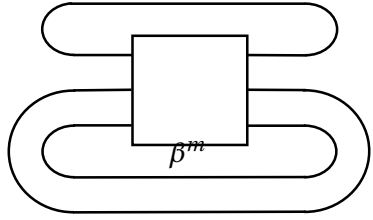}
\caption{Regular projection of $\widehat\be^m$}
\label{f:projection} 
\end{figure} 
Colour the regions of $\bR^2\setminus P$ alternately black and white, so that the unbounded region is white. The black 
regions determine a spanning surface $S$ for the knot $K^m$. Push the interior of $S$ into the 4--ball 
and let $M\to B^4$ be the associated 2--fold branched cover. Gordon and Litherland~\cite[Theorem~3]{GL78} 
give a recipe to compute the intersection form of $M$. Applying their recipe it is easy to check that the intersection 
lattice $H_2(M;\bZ)/\Tor$ is isomorphic to $\La_\Ga$. Clearly $\del M = -Y_K$, and the claim is proved when $d=0$.

{\em Second case: $d= 1$}. 
Let $T$ be an oriented, one--holed torus viewed as a 2--dimensional manifold with boundary and $\Diff^+(T,\del T)$ 
be the group of orientation--preserving diffeomorphisms of $T$ which restrict to the identity on the boundary. 

Let $Y_K$ be the 2--fold cover of $S^3$ branched along $K$, which  up to isotopy can be assumed 
to be positively transverse to the standard open book decomposition of $S^3$ 
with disk pages. Then, the pull--back of the standard open book on $S^3$ under the covering map $Y_K\to S^3$ has pages 
homeomorphic to a one--holed torus $T$ and, there are right--handed Dehn twists $\ga_1, \ga_2\in\Diff^+(T,\del T)$ 
along two simple closed curves in $T$ intersecting transversely once, such that the monodromy $h:T\to T$ is  
\[
h = (\ga_1\ga_2)^3 \ga_1^{x_1}\ga_2^{-y_1}\cdots\ga_1^{x_t}\ga_2^{-y_t}\in\Diff^+(T,\del T).
\]
In~\cite[\S 4]{Li14} it is shown, using the open book $(T,h)$, that $Y_K$ bounds a smooth, oriented, positive definite 
4--manifold $N$ such that the intersection lattice of $M=-N$ is isometric to $\La_\Ga$. 
\end{proof}

\begin{rem}[Pointed out by the referee]
It is possible to give another proof for the case $d=1$ of Lemma~\ref{l:2foldcover} using the same method as in the 
case $d=0$ by writing the knot $K$ as the almost--alternating closure of 
\[
\si_1^{x_1}\si_1^{-y_1}\si_2^{-x_2}\cdots\si_1^{x_t}\si_1^{-y_t+1}\si_2^{-x_1}\si_1^2\si_2\si_1^2,
\]
using the relation $(\si_1\si_2)^3 = (\si_2\si_1^2)^2$. This also gives an alternative proof of the following Lemma~\ref{l:negdef}: alternating diagrams have definite Goeritz matrices, almost--alternating diagrams have precisely 
one definite Goeritz, and one can easily check that the other Goeritz of the almost-alternating $K$ is indefinite.
\end{rem} 

\begin{lemma}\label{l:negdef}
The lattice $\La_\Ga$ is negative definite for any $t, x_i, y_i,\geq 1$ with $\sum_i y_i\geq 2$. 
\end{lemma} 

\begin{proof} 
Choose a `circular' order on the set $V$ of vertices of $\Ga$ and, for each $v\in V$, let $v'$ (respectively $v''$) 
be the vertex coming immediately before (respectively after) $v$. 
Let $\xi = \sum_{v\in V} a_v v \in\La_\Ga$, with $a_v\in\bZ$. Then, 
\[
\xi\cdot\xi = \left(\sum_{v\in V} a_v v \right)\cdot\left(\sum_{w\in V} a_w w\right) = 
\sum_{v\in V} \left(a_v a_{v'} v\cdot v' + a^2_v v\cdot v + a_v a_{v''} v\cdot v''\right). 
\]
Since $\sum_{v\in V} a_v a_{v'} v\cdot v' = \sum_{v\in V} a_v a_{v''} v\cdot v''$ and $v\cdot v\leq -2$ 
for each $v\in\V$, we have 
\begin{multline*}
\xi\cdot\xi = 
\sum_{v\in V} \left(2 a_v a_{v'} v\cdot v' + a^2_v v\cdot v\right) 
\leq \sum_{v\in V} \left(2 a_v a_{v'}v\cdot v' -2a^2_v\right) = \\
\sum_{v\in V} \left(2 a_v a_{v'}v\cdot v' - a^2_v - a^2_{v'}\right) = 
- \sum_{v\in V} (a_v-a_{v'}v\cdot v')^2\leq 0. 
\end{multline*}
Therefore, $\xi\cdot\xi=0$ implies $a_v=a_{v'}v\cdot v'$ for each $v\in V$. In particular, 
$a_v^2 = a_{v'}^2$ for each $v\in\V$. 
If we call $a^2$ this common value of the $a^2_v$'s, we  have 
\[
0 = \xi\cdot\xi = \sum_{v\in V} \left(2 a_v a_{v'} v\cdot v' + a^2_v v\cdot v\right) = 
\sum_{v\in V} \left(2 a^2_{v'} + a^2_v v\cdot v\right) = a^2 \sum_{v\in V} (v\cdot v + 2). 
\]
Since $\sum_{v\in V} (v\cdot v + 2) = -\sum_i x_i < 0$, this implies $a=0$, hence $a_v=0$ for 
each $v\in V$ and $\xi=0$.
\end{proof} 

\begin{prop}\label{p:embed}
Let $K\subset S^3$ be a knot of concordance order $k$ 
which is the closure of a 3--braid of the form given by Equation~\eqref{e:prelim}, 
with $d\in\{0,1\}$ and $\sum_{i=1}^t (x_i-y_i) = -4d$. Let $\Ga$ denote the 
weighted graph of Figure~\ref{f:graph} determined by the given values 
of $d$, $x_i$ and $y_i$. Then, the orthogonal sum $\La^k_\Ga$ of $k$ copies of $\La_\Ga$ 
embeds isometrically in the standard negative definite lattice $\bZ^N$, where $N=k\sum_i y_i$. 
Moreover, $\La^k$ has finite and odd index as a sublattice of $\bZ^N$. 
\end{prop} 

\begin{proof} 
Since $K=\hat\be$ has concordance order $k$, $K\#\stackrel{(k)}{\cdots}\# K$ bounds a properly 
and smoothly embedded disk $D\subset B^4$. The 2--fold cover $N_D\to B^4$ branched along $D$ 
is a rational homology ball~\cite[Lemma~2]{CG86} and $\del N_D = Y_K\#\stackrel{(k)}{\cdots}\# Y_K$. 
Let $M^k:=\natural^k M$ be the boundary--connect sum of $k$ copies of the 4--manifold $M$ of 
Lemma~\ref{l:2foldcover}. 
Then, $M^k$ has intersection lattice isomorphic to $\La^k_\Ga:= \La_\Ga\perp\stackrel{(k)}{\cdots}\perp\La_\Ga$. In view of Lemma~\ref{l:negdef}  
the smooth, closed, 4--manifold $X = N_D\cup M^k$ has negative definite intersection form, and 
by Donaldson's `Theorem A'~\cite[Theorem~1]{Do87} the intersection lattice $H_2(X;\bZ)/\Tor$ is isometric to the standard negative lattice $\bZ^N$, where 
$N=b_2(X)=b_2(M^k)=k\sum_i y_i$. 
This gives an isometric embedding $\La^k_\Ga\subset\bZ^N$. 
The fact that $\La^k_\Ga$ has finite and odd index follows, via standard arguments, from the fact that, 
since $K$ is a knot, the determinant of $K$ is odd. 
\end{proof}

\section{Lattice analysis}\label{s:latan}

\subsection{Circular subsets}\label{ss:prelim}

The standard negative lattice $(\bZ^N,-I)$ admits a basis 
$\E=\{e_1,\ldots, e_N\}$ such that the intersection product between $e_i$ and $e_j$ is 
$e_i\cdot e_j = -\de_{ij}$. Such a basis is unique up to permutations 
and sign reversals of its elements. From now on, we shall call any such basis $\E\subset\bZ^N$ 
a {\em canonical basis}, and denote the standard negative lattice simply $\bZ^N$. 
Given a subset $\V\subset\bZ^N$, we shall denote by $|\V|$ the cardinality of $\V$, and  
by $\La_\V$ the intersection lattice consisting of the subgroup of $\bZ^N$ generated by $\V$, 
endowed with the restriction of the intersection form on $\bZ^N$. 
Elements of the set $\V$ will be called indifferently {\em elements} or {\em vectors}. 
We will say that a vector $u\in\bZ^N$ {\em hits} a vector $v\in\bZ^N$ if $u\cdot v\neq 0$. 
On a finite subset $\V\subset\bZ^N$ 
is defined the equivalence relation $R$ generated by the reflexive and symmetric relation
given by $u\sim_Rv$ if and only if $u\cdot v\neq 0$. We call {\em connected components} 
of $\V$ its $R$--equivalence classes, and we call {\em connected} a subset consisting of 
a single $R$--equivalence class. 

\begin{defn}\label{d:circular}
A (not necessarily connected) subset $\V\subset\bZ^N$ is a {\em circular subset} if: 
\begin{itemize}
\item 
$|C|\geq 3$ for each connected component $C\subset\V$;
\item
$v\cdot w\in\{-1,0,1\}$ for any $v,w\in\V$ with $v\neq w$;
\item
for each $v\in\V$, the set $\{u\in\V\ |\ |u\cdot v|=1\}$ has two elements.
\end{itemize}
\end{defn}
Given a circular subset $\V\subset\bZ^N$, the vector $W:=\sum_{v\in\V} v \in \La_\V$ 
will be called the {\em Wu element} of $\V$. For any subset $\V\subset\bZ^N$, a canonical basis 
$\E\subset\bZ^N$ such that $\sum_{v\in\V} v = -\sum_{e\in\E} e$ will be called {\em adapted to $\V$}.

\begin{lemma}\label{l:wuelement} 
Let $\V\subset\bZ^N$ be a circular subset such that $\La_\V\subset\bZ^N$ has finite and odd index and  
whose Wu element $W$ satisfies $W\cdot W = -N$. Then, there is a canonical basis $\E\subset\bZ^N$ 
adapted to $\V$. 
In particular, for each $e\in\E$ we have
\begin{equation}\label{e:sumofcoeff}
\sum_{v\in\V} v\cdot e =  1. 
\end{equation}
\end{lemma} 

\begin{proof} 
Since the index of $\La_\V$ is finite and odd, given $v\in\bZ^N$ there exists an odd integer $d$ 
such that $dv\in\La_\V$. It is easy to check that $W\in\La_\V$ is characteristic in $\La_\V$, therefore $W\cdot (dv)$ and $(dv)\cdot (dv)$ are congruent modulo $2$. But, since $d$ is odd, the first number is congruent to $W\cdot v$ and the second one to $v\cdot v$. This shows that $W$ is a characteristic vector when viewed in $\bZ^N$. In particular, 
given any canonical basis $\E\subset\bZ^N$ we have $W\cdot e \neq 0$ for each $e\in\E$. Since $W\cdot W= -N$, 
we have $W\cdot e = \pm 1$ for each $e$, and up to reversing the signs of some 
of the $e$'s we have $W=-\sum_{e\in\E} e$. Equation~\eqref{e:sumofcoeff} follows immediately 
using the definition of $W$. 
\end{proof} 

\begin{lemma}\label{l:coefficients}
Let $\V\subset\bZ^N$ be a circular subset and $\E\subset\bZ^N$ a canonical basis adapted to $\V$. 
Then, for every $v\in\V$ and $e\in\E$ we have $v\cdot e \in\{-1,0,1,2\}$, and there is a map
\[ 
\W:=\{v\in\V\ |\ W\cdot v = v\cdot v+2\}\to\E 
\]
defined by sending each $v\in\W$ to the unique $e_v\in\E$ such that 
$v\cdot e_v\in\{-1,2\}$.
\end{lemma} 

\begin{proof} 
We can write each $v\in\V$ as a linear combination 
\[
v = -\sum_{e\in\E} (v\cdot e) e.
\]
For each $v\in\V$ we have the equality 
\[
-v\cdot v + v\cdot W = v\cdot (-v+W) = \sum_{e\in\E} v\cdot e(v\cdot e-1).
\]
By the second and third conditions in Definition~\ref{d:circular}, we have 
\[
-2 \leq -v\cdot v + v\cdot W \leq 2
\]
for each $v\in\V$.  The quantity $x(x-1)$ is always nonnegative when $x\in\bZ$, so 
we conclude 
\begin{equation}\label{e:coefficients}
\sum_{e\in\E} v\cdot e (v\cdot e - 1) =
\begin{cases} 
2\quad\text{if}\ v\in\W\\
0\quad\text{if}\ v\not\in\W.
\end{cases}
\end{equation}
By the second and third condition in Definition~\ref{d:circular}, 
Equation~\eqref{e:coefficients} implies $v\cdot e\in\{-1,0,1,2\}$ for every 
$e$. Moreover, if $v\in\W$ there is exactly one $e_v\in\E$ such that $v\cdot e_v\in\{-1,2\}$. 
This gives the statement and concludes the proof.
\end{proof}

\subsection{Semipositive circular subsets}\label{ss:semiposcs}

A circular subset $\V\subset\bZ^{N}$ will be called {\em semipositive} if each connected component 
$C\subset\V$ contains a single pair of vectors $u, v\in C$ such that $u\cdot v=-1$.

\begin{lemma}\label{l:semipcs-u}
Let $\V\subset\bZ^{|\V|}$ be a semipositive circular subset such that $v\cdot v\leq -2$ for each $v\in\V$ and 
$\sum_{v\in C}(v\cdot v+2)<-1$ for each connnected component $C\subset\V$. Let $\E$ be 
a canonical basis adapted to $\V$. Then, for each $e\in\E$ not in the image of the map 
of Lemma~\ref{l:coefficients}, there exists $u\in\V$ such that $u\cdot u=-2$, $u\cdot e=1$ 
and $u$ is the only vector of $\V$ which hits $e$. 
\end{lemma} 

\begin{proof}
Let $e\in\E$ be an element not in the image of the map of Lemma~\ref{l:coefficients}. 
Then, we have $v\cdot e\in\{0,1\}$ for each $v\in\V$. In view of Equation~\eqref{e:sumofcoeff}  
there exists a unique $u\in\V$ such that $u\cdot e=1$. The subset $\V'=\V\setminus\{u\}\cup\{u+e\}$  
is semipositive circular, contained in the span of $\E\setminus\{e\}$, and satisfies 
$\sum_{v\in C}(v\cdot v+2)<0$ for each connected component $C\subset\V'$. 
If $v\cdot v\leq -2$ for each $v\in\V'$, then $\La_{\V'}$ is isometric to $\La_\Ga$ 
for some $t, x_i, y_i\geq 1$. By  Lemma~\ref{l:negdef} this would imply 
$\rk(\La_{\V'}) = |\V'|$. Since $\V'$ can be viewed as a subset of $\bZ^{|\V|-1}$ and  
$|\V'| = |\V|$, we must have $(u+e)\cdot (u+e) = -1$,  
which implies $u\cdot u=-2$. 
\end{proof} 

\begin{lemma}\label{l:(-1)-contr}
Let $\V\subset\bZ^N$ be a semipositive circular subset, and let $\E\subset\bZ^N$ 
be a canonical basis adapted to $\V$. Let $u\in\V$ with $u\cdot u=-1$ and suppose that the connected 
component containing $u$ has cardinality at least $4$. Let $u', u''\in\V$ be the two distinct vectors 
such that $|u'\cdot u|=|u''\cdot u| = 1$. Then, the subset 
\[
\V' = \V\setminus\{u\}\cup\{u'+(u'\cdot u) u, u'' + (u''\cdot u) u\}
\]
is a semipositive, circular subset contained in the span of $\E' = \E\setminus\{e\}$, 
where $e\in\E$ is the only element such that $e\in\{u, -u\}$, and $\E'$ is adapted to $\V'$.
\end{lemma}

\begin{proof} 
The proof is an easy exercise left to the reader. 
\end{proof} 

\begin{prop}\label{p:semipos}
Let $\V\subset\bZ^N$ be a semipositive circular subset such that: 
\begin{itemize}
\item
$\La_\V\subset\bZ^N$ has finite and odd index;
\item
its Wu element $W=\sum_{v\in\V} v$ satisfies $W\cdot W = - N$; 
\item 
$v\cdot v\leq -2$ for each $v\in\V$; 
\item
$\sum_{v\in D}(v\cdot v+2)=4-|D|<-1$ for each connected component $D\subset\V$.
\end{itemize} 
Let $C = \{v_1,\ldots, v_{|C|}\}\subset\V$ be any connected component of $\V$,  
and $(-c_1,\ldots, -c_{|C|})$ the corresponding string of self--intersections. 
Then, there is an iterated blowup $(s_1,\ldots, s_{|C|})$ of $(0,0)$ containing exactly two $1$'s, 
and such that $c_i\geq s_i$ for each $i=1,\ldots, |C|$. 
\end{prop} 

\begin{proof} 
By Lemma~\ref{l:wuelement} there is a canonical basis $\E\subset\bZ^N$ adapted to $\V$, and Lemma~\ref{l:coefficients} applies.
We observe that the subset $\W\subset\V$ defined in Lemma~\ref{l:coefficients} does not coincide with $\V$. 
In fact, it is easy to check that if $k$ is the number of connected components of $\V$ then 
$|\W|=|\V|-2k$. Since $\La_\V\subset\bZ^N$ has finite index we have $|\V|=|\E|$, therefore 
there are at least $2k$ distinct elements in the complement of the image of the map 
$\W\to\E$ of Lemma~\ref{l:coefficients}. As a result, there are at least $2k$ distinct vectors $u\in\V$ as in 
Lemma~\ref{l:semipcs-u}. We modify all of those vectors by replacing each $u$ with $u+e$, 
where $e$ is the associated vector of $\E$. The result is a new subset $\V'$ with $|\V'| = |\V|$, 
and at least $2k$ vectors of square $-1$. Let $\E'\subset\E$ be the subset 
obtained by erasing from $\E$ all the vectors $e$ corresponding to the $u$'s that were modified. 
Then, $\V'$ is contained in the span of $\E'$ and $\E'$ is adapted to $\V'$, i.e.~$\sum_{v\in\V'} v = -\sum_{e\in\E'} e$. 
Lemma~\ref{l:(-1)-contr} can then be applied several times. We apply the lemma as many times as 
possible, i.e.~until the resulting subset $\V''$ has no vectors of square $-1$ belonging to a connected 
component with at least four vectors. We claim that every connected 
component of $\V''$ contains three vectors. In fact, a connected component $C''\subset\V''$ with no $(-1)$--vectors would 
contain two vectors $v$ such that $W\cdot v=v\cdot v$, and Equations~\eqref{e:sumofcoeff} and~\eqref{e:coefficients} 
would imply that each such $v$ would be the unique vector of $\V''$ hitting some $e\in\E$. But we had eliminated from 
$\V$ all $e$'s hitting a single vector at the beginning, and the construction of $\V''$ from $\V$ implies that a vector of 
$\V''$ can have this property only if it has square $-1$. 
Therefore $C''$ must consist of three vectors, and the claim holds. Next, we claim that 
each connected component of $\V''$ consists of $(-1)$--vectors. To see this it suffices to show that 
$\sum_{v\in\V''} (v\cdot v+2) = 3k$, since each of the $k$ components can contribute at most $3$ to this quantity.  
By our assumptions on $\V$ we have 
\[
\sum_{v\in\V} (v\cdot v+2) = \sum_{C\subset\V} (4-|C|) = 4k - |\V|.  
\]
In the first step of the above construction we turn a certain number $m\geq 2k$ of $(-2)$--vectors 
into $(-1)$--vectors, therefore $\sum_{v\in\V'} = 4k-|\V|+m$. Then we apply Lemma~\ref{l:(-1)-contr} 
$\sum_{C\subset\V} (|C|-3) = |\V|-3k$ times, each time increasing the quantity by $1$. The result is 
$k+m$. This number is at most $3k$ because each component can contribute at most $3$. On the other 
hand, $m\geq 2k$ by construction. This forces $m=2k$, and the last claim is proved. If we now look at 
how $\V'$ is obtained from $\V''$, we see that, in view of Equations~\eqref{e:sumofcoeff} and~\eqref{e:coefficients}, 
each component of $\V'$ must contain at least two 
$(-1)$--vectors. Since $m=2k$, each component of $\V'$ contains exactly two $(-1)$--vectors. It is now easy 
to check that this implies the statement. 
\end{proof}

\begin{thm}\label{t:semipos} 
Let $\Ga$ be a weighted graph as in Figure~\ref{f:graph} with $d=1$ and $\sum_i y_i = \sum_i x_i + 4\geq 6$. 
Let $N=k\sum_i y_i$ for some $k\geq 1$ and suppose that there is an isometric 
embedding $\La^k_\Ga\subset\bZ^N$ of finite and odd index. Then, the string of positive integers 
$
(x_1+2,\overbrace{2,\ldots,2}^{y_1-1},\ldots,x_t+2,\overbrace{2,\ldots,2}^{y_t-1})
$
is obtained from an iterated blowup of $(0,0)$ by replacing two $1$'s with $2$'s. 
\end{thm} 

\begin{proof} 
The lattice $\La^k_\Ga$ has a natural basis, whose elements are in 1--1 correspondence with the 
vertices of the disjoint union of $k$ copies of $\Ga$, and intersect as prescribed by edges and weights. 
Since $d=1$, the image of such a basis is a semipositive circular subset $\V\subset\bZ^N$. 
The Wu element $W=\sum_{v\in\V} v$ satisfies 
\[
W\cdot W = k\left(\sum_i (-2-x_i) -2\sum_i (y_i-1) + 2\sum_i y_i - 4\right) = 
-k\sum_i y_i = -N
\]
and since $\La_\V=\La^k_\Ga$, the embedding $\La_\V\subset\bZ^N$ has finite and odd index.
Moreover, each connected component $D\subset\V$ satisfies 
\[
\sum_{v\in D} (v\cdot v+2) = -\sum_i x_i =  4 - \sum_i y_i  = 4 - |D| \leq -2.
\]
Therefore we can apply Proposition~\ref{p:semipos} to conclude that for any connected component $C\subset\V$ 
there is an iterated blowup $(s_1,\ldots, s_{|C|})$ of $(0,0)$ containing exactly 
two $1$'s such that the elements of the string 
$
(c_1,\ldots, c_{|C|}) := (x_1+2,\overbrace{2,\ldots,2}^{y_1-1},\ldots,x_t+2,\overbrace{2,\ldots,2}^{y_t-1})
$
satisfy $c_i\geq s_i$ for $i=1,\ldots, |C|$. It only remains to check that $c_i = s_i+1$ if and only if $s_i=1$.  
This follows immediately from the fact that 
\[
\sum_i c_i = -\sum_{v\in C} v\cdot v = 3|C| - 4= \sum_i s_i + 2, 
\]
where the last equality can be easily established by induction on the number of blowups. 
\end{proof}

\subsection{Positive circular subsets}\label{ss:poscs}

A circular subset $\V\subset\bZ^{N}$ is called {\em positive} if $u\cdot v \geq 0$ for any $u,v\in\V$ with $u\neq v$. 
Observe that for a positive circular subset $\V$ the subset $\W\subset\V$ defined in 
Lemma~\ref{l:coefficients} coincides with $\V$. Moreover, two canonical bases adapted to $\V$ differ 
by a permutation of their elements. This easily implies that, whether the map 
$\W\to\E$ of Lemma~\ref{l:coefficients} is injective or not, does not depend on the choice of $\E$.
We will assume these facts as understood throughout this subsection. 

The analysis of positive circular subsets admitting an adapted canonical basis splits naturally into two subcases, 
according to whether the map of Lemma~\ref{l:coefficients} is injective or not. We first deal with the simplest 
case, when the map is not injective. 

\subsection*{First subcase}\label{ss:notinjsubcase}

Let $\V\subset\bZ^N$ be a positive circular subset admitting an adapted canonical basis. 
From now on, and until further notice, we will assume that the map of Lemma~\ref{l:coefficients} is not injective. 

\begin{lemma}\label{l:pcsnotinj} 
Let $\V\subset\bZ^N$ be a positive circular subset and $\E\subset\bZ^N$ a canonical basis adapted to $\V$. 
Suppose that: 
\begin{itemize}
\item 
$|\V|=N$;
\item
the map of Lemma~\ref{l:coefficients} is not injective;
\item
$v\cdot v\leq -2$ for each $v\in\V$; 
\item
$\sum_{v\in D} (v\cdot v + 2) = -|D|$ for each connected component $D\subset\V$. 
\end{itemize}
Then, for each $e\in\E$ not in the image of the map of Lemma~\ref{l:coefficients} 
there exists $u\in\V$ such that $u=e_u-e$ and 
$\{v\in\V\ |\ v\cdot e=1\} = \{u\}$. Moreover, let $C\subset\V$ be the connected component 
containing $u$. Then, 
\begin{itemize} 
\item
if $|C|>3$ there is $v\in C$ such that $v\cdot e_u=1$ and 
$v$ hits $\geq 3$ distinct vectors of $\E$;
\item
if $|C|=3$ we have $C = \{e_1-e_2,e_3-e_1,-2e_3-e_1\}$ for some $e_1,e_2,e_3\in\E$.  
\end{itemize}
\end{lemma} 

\begin{proof}
By Lemma~\ref{l:negdef} and the first assumption we have $\rk\La_\V = |\V| = N$. Therefore, since the map of Lemma~\ref{l:coefficients} is not injective, it cannot be surjective. 
For each $e\in\E$ not in the image of the map  we have 
$v\cdot e\in\{0,1\}$ for each $v\in\V$. By Equation~\eqref{e:sumofcoeff} there is a unique $u\in\V$ such 
that $u\cdot e=1$. We claim that the set 
$\{e'\in\E\ |\ u\cdot e'\neq 0\}$ contains only $e$ and $e_u$ and moreover $u\cdot e_u=-1$. 
Otherwise, we would have $u\cdot u<-2$, and by replacing $u$ in $\V$ 
with $u + (u\cdot e) e$ we would obtain a new set $\V'\subset\bZ^N$ of 
vectors contained in the span of $\E\setminus\{e\}$, such that $\La_{\V'}$ would be isometric to 
$\La_\Ga$ for some $t, x_i, y_i\geq 1$. Then, by Lemma~\ref{l:negdef} we would have $\rk(\La_{\V'}) = |\V'|$.
This would contradict the fact that $|\V'|=|\V|=|\E|>|\E\setminus\{e\}|$, therefore the claim is proved, and  
we have $u=e_u-e$. Let $v,w\in\V$ be the two vectors satisfying $v\cdot u=w\cdot u=1$. Then, 
$v\cdot e_u=w\cdot e_u=1$, and we may write 
\begin{equation}\label{e:notinj}
v = a_v e_v - e_u - \cdots, \quad w = a_w e_w - e_u - \cdots\quad \text{for some}\ a_v,a_w\in\{-2,1\}.
\end{equation}
Let $C\subset\V$ be the connected component containing $u$. By contradiction, suppose that 
$|C|>3$ and both $v$ and $w$ hit only two vectors of $\E$. Then,
\[
0 = v\cdot w = (a_v e_v - e_u)\cdot(a_w e_w - e_u) = a_v a_w (e_v\cdot e_w) - 1,
\]
which is impossible because $e_v\cdot e_w\in\{-1,0\}$. Therefore, up to renaming $v$ and $w$ 
we may assume that $v$ hits at least three distinct vectors of $\E$. If $|C|=3$ we also have $v\cdot w=1$, 
therefore by Equation~\eqref{e:notinj} 
\[
1 = v\cdot w = (a_v e_v - e_u-\cdots)\cdot(a_w e_w - e_u-\cdots) 
\leq (a_v e_v - e_u)\cdot(a_w e_w - e_u) = a_v a_w (e_v\cdot e_w) - 1,
\]
which implies $e_v=e_w$ and $a_v a_w = -2$. Up to swapping $v$ and $w$ 
we may assume $a_v=1$ and $a_w = -2$. This implies $v\cdot v\leq -2$ and $w\cdot w\leq -5$, 
and since $-3 = \sum_{v\in C} (v\cdot v+2)$ we must have $v\cdot v=-2$ and $w\cdot w=-5$, 
and the statement follows immediately.
\end{proof}

\begin{defn}\label{d:-2-exp}
A string of negative integers $(-n_1,\ldots, -n_{k+1})$ is obtained from $(-m_1,\ldots, -m_k)$ 
by a {\em (-2)-expansion} if $m_1=2$ and 
\[
(-n_1,\ldots, -n_{k+1}) =
\begin{cases} 
(-2,-m_1,\ldots, -m_k-1)\quad\text{or}\\
(-2, -m_2-1,\ldots, -m_k, -m_1).
\end{cases}
\]
\end{defn} 

\begin{prop}\label{p:pcsnotinj} 
Let $\V\subset\bZ^N$ be a positive circular subset such that: 
\begin{itemize} 
\item
$v\cdot v\leq -2$ for each $v\in\V$; 
\item
$\La_\V\subset\bZ^N$ has finite and odd index; 
\item  
the Wu element $W=\sum_{v\in\V} v$ satisfies $W\cdot W = -N$; 
\item
the map of Lemma~\ref{l:coefficients} is not injective; 
\item
$\sum_{v\in D} (v\cdot v + 2) = -|D|$ for each connected component $D\subset\V$. 
\end{itemize}
Then, there is a connected component $C\subset\V$ whose string of self--intersections  
is obtained, up to circular symmetry, by a sequence of (-2)-expansions from $(-2, -2, -5)$ . 
\end{prop} 

\begin{proof} 
Observe that $\La_\V$ is isometric to an intersection lattice $\La_\Ga$ for some $t, x_i, y_i\geq 1$. 
By Lemma~\ref{l:negdef} and the second assumption we have $\rk\La_\V = |\V| = N$. Therefore, since the map of 
Lemma~\ref{l:coefficients} is not injective, it cannot be surjective. If a connected component of $\V$ has exactly 
three elements the statement follows immediately from Lemma~\ref{l:pcsnotinj}. Hence, suppose that 
each connected component of $\V$ has at least four elements. By Lemma~\ref{l:wuelement}, the first two assumptions 
imply the existence of a canonical basis $\E$ adapted to $\V$. This, together with the remaining assumptions imply that 
Lemma~\ref{l:pcsnotinj} can be applied to $\V$. Let $u$ and $v$ be vectors as in 
Lemma~\ref{l:pcsnotinj} belonging to a connected component $C\subset\V$. 
We can change $\V$ into a new subset $\V'$ by a process we will call {\em contraction}.  
The subset $\V'$ is obtained by replacing $C$ with $C':=\left(C\setminus\{u,v\}\right)\cup\{v+e_u\}$,  
which is contained in the span of $\E' := \E\setminus\{e\}$. When regarded as a subset of $\bZ^{N-1}$, 
$\V'$ is still a positive circular subset, $\E'$ is adapted to $\V'$, $|\V'|=N-1$, $v'\cdot v'\leq -2$ for each $v'\in\V'$ and  
$\sum_{v'\in C'} (v'\cdot v' + 2) = -|C'|$. Moreover, $e_u$ does not belong to the image of the map of Lemma~\ref{l:coefficients} and $|C'|\geq 3$. If $|C'|>3$ we can apply Lemma~\ref{l:pcsnotinj} again 
and contract $\V'$ to a subset $\V''$, possibly modifying a connected component different from $C'$. 
We can keep contracting as long as all connected components of the resulting subset 
have cardinality greater than $3$. When one of the components reaches cardinality $3$ 
we can apply the last part of Lemma~\ref{l:pcsnotinj}. The statement is easily obtained combining 
the information from the lemma with the fact that the component is the result of a sequence of 
contractions. 
\end{proof}

\subsection*{Second subcase}\label{sss:injsubcase}
In this subsection we study the positive circular subsets $\V\subset\bZ^N$ with an adapted canonical 
basis such that the map of Lemma~\ref{l:coefficients} is injective. 

\begin{lemma}\label{l:pcsinj}
Let $\V\subset\bZ^{|\V|}$ be a positive circular subset such that: 
\begin{itemize} 
\item
$v\cdot v\leq -2$ for each $v\in\V$; 
\item  
the Wu element $W=\sum_{v\in\V} v$ satisfies $W\cdot W = -|\V|$; 
\item
there is a canonical basis $\E\subset\bZ^{|\V|}$ adapted to $\V$ and the 
map of Lemma~\ref{l:coefficients} is injective. 
\end{itemize}
Then, the following properties hold:
\begin{enumerate}
\item
for each $v\in\V$ we have $v\cdot e_v=-1$; 
\item
for each $e\in\E$ there exist distinct elements $u_e, v_e, w_e\in\V$ such that $u_e\cdot e=-1$ and 
$v_e\cdot e = w_e\cdot e = 1$. Moreover, $x\cdot e=0$ for each $x\in\V\setminus\{u_e,v_e,w_e\}$;
\item
for each $u\in\V$ with $u\cdot u=-2$, there exist $f,g\in\E$ and $v, w, z\in\V$ such that: 
\begin{itemize}
\item
$u=e_u - f$; 
\item
$v\cdot u=1$, $v\cdot f = 0$ and $v = e_v - e_u - \cdots$ ;
\item
$w\cdot u = 1$, $w\cdot e_u=0$ and $w = f - \cdots$;
\item 
$z\cdot u = 0$ and $z = e_z - e_u - f - g -\cdots$.    
\end{itemize} 
\end{enumerate}
\end{lemma} 

\begin{proof}
(1) By the injectivity of the map each $v\in\V$ is characterized as the unique element $u\in\V$ 
such that $u\cdot e_v\in\{-1,2\}$. It follows that for each $u\in\V\setminus\{v\}$ we have $u\cdot e_v\in\{0,1\}$. Clearly, Equation~\eqref{e:sumofcoeff} can be satisfied only if $v\cdot e_v=-1$. 

(2) Since $|\V|=|\E|$, the 
map $v\mapsto e_v$ is a bijection. Denote by $e\mapsto u_e$ the inverse map. By (1) $u_e\cdot e = -1$, and 
by Equation~\eqref{e:sumofcoeff} there exist distinct elements $u_e, v_e, w_e\in\V$ such that $u_e\cdot e=-1$ and 
$v_e\cdot e = w_e\cdot e = 1$, while $x\cdot e=0$ for each $x\in\V\setminus\{u_e,v_e,w_e\}$. 

(3) The fact that $u=e_u-f$ for some $f\in\E$ follows immediately from (1), the definition of $e_u$ and the 
fact that $u\cdot u = -2$. Observe that for any $v\in\V$ with $v\cdot u=1$, since 
$v\cdot e\in\{-1,0,1\}$ for each $e\in\E$, we have 
$(v\cdot e_u, v\cdot f)\in\{(1,0),(0,-1)\}$. There are exactly two elements $v, w\in\V$ such that 
$u\cdot v=u\cdot w=1$, and by the previous observation we have $v\cdot f, w\cdot f \in\{-1,0\}$. On the other 
hand, $v\cdot f=w\cdot f =-1$ would imply $e_v=f=e_w$, which is impossible because we are 
assuming that the map $u\mapsto e_u$ is injective. Therefore, either $(v\cdot e_u, v\cdot f) = (1,0)$ 
or $(w\cdot e_u, w\cdot f) = (1,0)$, and up to renaming $v$ and $w$ we may assume  
$(v\cdot e_u, v\cdot f) = (1,0)$. Therefore $v\cdot f = 0$ and $v = e_v - e_u - \cdots$. 
This forces $(w\cdot e_u, w\cdot f) = (0,-1)$, because otherwise $w\cdot e_u = 1$, and there 
would be $w'\in\V$ distinct from $v$ and $w$ with  
$e_{w'} = f$. Since $w'\cdot u=0$, this would imply $w'\cdot e_u=1$, contradicting (2) because 
$e_u$ would already appear in $u$, $v$ and $w$. Therefore $w\cdot e_u=0$ and $w=f-\cdots$. 

By (2) and the fact that $(v\cdot e_u, w\cdot e_u) = (-1,0)$, there exist $z\in\V$ such that $z\cdot e_u = -1$. Since 
the vectors adjacent to $u$ are $v$ and $w$, we have $z\cdot u=0$ and therefore $z\cdot f = -1$ as well. 
Thus, $z = e_z - e_u - f-\cdots$ and $z\cdot z\leq -3$. We claim that $z\cdot z<-3$. In fact, suppose by 
contradiction that $z\cdot z = -3$, so that $z = e_z - e_u - f$. Since $e_u$ and $f$ already appear three times, $e_z$ must appear in both the adjacent vectors of $z$, say $z'$ and $z''$. On the other hand, we have 
$1\geq w\cdot z = (f-\cdots)\cdot (e_z-e_u-f) \geq 1$, therefore $w\cdot z=1$, which implies $w\in\{z', z''\}$. 
But this is not possible because $e_z$ cannot appear in $w$. Therefore, there exist some $g\in\E$ with 
$z = e_z-e_u-f-g-\cdots$. 
\end{proof}

\begin{defn}\label{d:-2-contr}
Let $\V\subset\bZ^{|\V|}$ be a positive circular subset satisfying the hypotheses of Lemma~\ref{l:pcsinj}.
If $u\in\V$ satisfies $u\cdot u=-2$ and the connected component of $u$ contains more than $3$ elements, 
by Lemma~\ref{l:pcsinj} there exist $f, g\in\E$ and $v, z\in\V$ such that $u=e_u-f$, $v=e_v-e_u-\cdots$ 
and $z = e_z - e_u - f -g-\cdots$. Then, we define $v' := u+v = e_u - f-\cdots$, $z' := z + e_u = e_z-f-g-\cdots$ 
and 
\[
\V' := \V\setminus\{u,v,z\}\cup \{v',z'\}.
\]
We say that the set  $\V'$ is obtained from $\V$ by a {\em (-2)--contraction}. 
\end{defn}

\begin{rmk}\label{r:-2-contr}
Observe that if every connected component $C$ of the set $\V'$ in Definition~\ref{d:-2-contr} satisfies 
$|C|\geq 3$, then $\V'$ is circular when regarded as a subset of the intersection lattice $\bZ^{|\V'|}$ 
spanned by $\E' = \E\setminus\{e_u\}$, and $v\cdot v\leq -2$ for each $v\in\V'$. A simple calculation shows that 
the Wu element $W' = \sum_{v\in\V'} v$ satisfies $W'\cdot W' = W\cdot W + 1 = -|\V'|$ and $W'=-\sum_{e\in\E'} e$. 
In particular, $\E'$ is a canonical basis adapted to $\V'$. Moreover, for each $v\in\V'$ there is an 
$e\in\E'$ such that $v\cdot e=-1$. Therefore the map $\V'\to\E'$ defined in Lemma~\ref{l:coefficients} 
is surjective, hence also injective. This shows that $\V'$ satisfies all the assumptions of Lemma~\ref{l:pcsinj}. 
\end{rmk}

If, after applying a $(-2)$--contraction to a positive circular subset $\V$ satisfying the assumptions 
of Lemma~\ref{l:pcsinj}, we obtain a new positive circular subset $\V'$ which still contains a (-2)--vector whose 
connected component has more than $3$ elements, by Remark~\ref{r:-2-contr} we can apply a (-2)--contraction again, 
obtaining a set $\V''$, again satisfying the assumptions of Lemma~\ref{l:pcsinj}, and so on. 
Furthermore, at each step a $(-2)$--vector is eliminated but, in view of Definition~\ref{d:-2-contr}, 
no new $(-2)$--vector is created. Clearly, after a finite number of $(-2)$--contractions we end up with a 
circular subset $\Z\subset\bZ^{|\Z|}$ satisfying the assumptions 
of Lemma~\ref{l:pcsinj} and having the property that each $(-2)$--vector of $\Z$ belongs to 
a connected component with exactly three elements. 

In order to understand the set $\V$ we need more information about the set $\Z$. 
In the following, we shall denote by $\F\subset\bZ^{|\Z|}$ a canonical basis adapted to $\Z$. 
For each $u\in\Z$, we define 
\[
\F_u := \{f\in\F\ |\ u\cdot f\neq 0\}\subset\F.
\]
Following~\cite{Li07-1}, we consider the equivalence relation on $\Z$ generated by the relation 
given by $u\sim v$ if and only if $\F_u\cap\F_v\neq\emptyset$ and 
we call {\em irreducible components} of $\Z$ the resulting  
equivalence classes. We shall now analyze $\Z$ considering separately the two cases 
$t=2$ and $t\geq 3$, where $t$ is the parameter appearing in Equation~\eqref{e:prelim} 
(when $t=1$ the analysis is not necessary, as will be shown in the proof of Theorem~\ref{t:pcs}). 

\medskip
\noindent\underline{Suppose $t=2$.} 
By the considerations following Remark~\ref{r:-2-contr}, it is easy to see that $\Z$  
consists of some number $n$ of connected components $C_1,\ldots, C_n$, where each component $C_i$ contains three 
vectors $u_i$, $v_i$ and $w_i$ satisfying $u_i\cdot v_i = v_i\cdot w_i = w_i\cdot u_i = 1$, 
$u_i\cdot u_i =-2$, $v_i\cdot v_i = -2-a_i$ and $w_i\cdot w_i = -2 - b_i$, with $a_i, b_i\geq 1$.   

\begin{lemma}\label{l:sumofais+bis}
$\sum_{i=1}^n (a_i+b_i) = 3n$.
\end{lemma} 

\begin{proof} 
We refer to the notation of Figure~\ref{f:graph}. 
Each connected component of the original positive circular subset $\V$ contains 
$y_1 + y_2 - 2 = x_1 +x_2 - 2$ vectors of square $-2$. Therefore, the number of $(-2)$--contractions 
applied to each connected component of of $\V$ to obtain $\Z$ is $x_1 + x_2 - 3$. Each time we apply
a $(-2)$--contraction the self--intersection of some vector $w$ with $w\cdot w\leq -4$ increases by $1$. 
This shows that 
\[
\sum_{i=1}^n (a_i+b_i) = n (x_1+x_2) - n (x_1+x_2 - 3) = 3n. 
\] 
\end{proof} 

\begin{lemma}\label{l:ais+bis}
$a_i + b_i= 3$ for each $i=1,\ldots,n$. 
\end{lemma} 

\begin{proof} 
In view of Lemma~\ref{l:sumofais+bis}, it suffices to show that $a_i+b_i\geq 3$ for each $i=1,\ldots, n$. We do this 
by showing that the case $a_i+b_i=2$ cannot occur. Arguing by contradiction, suppose that 
$a_i + b_i = 2$ for some $i$. Then, $a_i=b_i=1$ and, up to renaming the elements of $\F$ we must have 
$u_i = f_1-f_2$, $v_i = f_2-f_3-f_4$ and $w_i = f_3-f_1-f_5$ for some $f_1,\ldots, f_5\in\F$. 
Let $v$ be the unique vector of $\F\setminus C_i$ such that $v\cdot f_1 = 1$. Since $v\cdot u_i=0$, 
we must have $v = e_v -f_1-f_2-\cdots$. Since $v\cdot w_i = 0$, we have either (i) $v = f_5 - f_1 - f_2-\cdots$ or 
(ii) $v = e_v -f_1 - f_2 - f_3-\cdots$. In Case (ii), consider the unique vector $w\in\Z\setminus C_i$ 
with $e_w = f_4$. Since $w\cdot v_i=0$, $\F_w$ should contain either $f_2$ or $f_3$. But $f_2$ already 
appears in $u_i$, $v_i$ and $v$, and $f_3$ in $v_i$, $w_i$ and $v$. The only possibility is that $w=v$ 
and $v=f_4-f_1-f_2-f_3-\cdots$, but this is incompatible with $v\cdot v_i=0$, therefore Case (ii) cannot occur. 
In Case (i), since $v\cdot v_i=0$ we must have $v = f_5 - f_1-f_2-f_4-\cdots$. Let $w\in\Z\setminus C_i$ be 
the unique vector with $e_w = f_4$. Then, since $w\cdot v_i=0$ we must have $w=f_4-f_3-\cdots$ and 
since $w\cdot w_i = 0$ we must have $w = f_4-f_3-f_5-\cdots$. Since $v\cdot w\leq 1$, we must have 
$v = f_5-f_1-f_2-f_4-f_6-\cdots$ and $w=f_4-f_3-f_5-f_6-\cdots$ for some $f_6\in\F$. Now consider 
the new set $\Z'$ obtained from $\Z$ by eliminating the vectors $u_i$, $v_i$ and $w_i$ and replacing $v$ 
with $v' = v+f_1+f_2+f_4$ and $w$ with $w' = w+f_3-f_4$. By construction, the set $\Z'$ is a positive circular subset 
of the span of $\F\setminus\{f_1,f_2,f_3,f_4\}$, and $|\Z'|=|\Z|-3$. Moreover, each connected component of $\Z'$ 
contains a vector with square $\leq -3$. This contradicts Lemma~\ref{l:negdef}, showing that Case (i) 
cannot occur either.
\end{proof} 

\begin{prop}\label{p:Zcaset=2}
When $t=2$ the set $\Z$ has an even number of connected components. Indeed, each 
irreducible component of $\Z$ is the union of two connected components $C_1$ and $C_2$, 
with 
\[
C_1 = \{f_1-f_2, f_2-f_3-f_4, f_4-f_5-f_6-f_1\}\quad\text{and}\quad C_2 = \{f_3-f_4-f_6, f_6-f_5, f_5-f_1-f_2-f_3\}
\]
for some $f_1,\ldots, f_6\in\F$.
\end{prop} 

\begin{proof} 
Let $S\subset\Z$ be an irreducible component, and let $C_1\subset S$ be one of its connected components. 
In view of Lemma~\ref{l:ais+bis}, up to permuting the elements of $\F$ we may assume that 
\[
C_1 = \{u_1=f_1-f_2, v_1=f_2-f_3-f_4, w_1=f_4-f_5-f_6-f_1\}. 
\]
Let $w_2\in S\setminus C_1$ be the unique vector with $w_2\cdot f_1=1$. Since $w_2\cdot u_1=w_2\cdot v_1=0$, 
we have $w_2\cdot w_2=-4$ and either (i) $w_2 = e_{w_2} - f_1- f_2 - f_3$ or (ii) 
$w_2 = e_{w_2} - f_1-f_2-f_4$. Case (ii) cannot occur, because if it did $f_2$ and $f_4$ would already have 
been used three times each, therefore the unique vector $v_2\in S\setminus C_1$ such that 
$e_{v_2} = f_3$ could not satisfy $v_2\cdot v_1=0$.

In Case (i), since $w_2\cdot w_1=0$, up to swapping $f_5$ and $f_6$ 
we have $w_2 = f_5 - f_1 - f_2 -f_3$. Let $v_2\in S\setminus C_1$ be the unique vector with $e_{v_2} = f_3$.
Since $v_2\cdot v_1=0$, we have $v_2 = f_3 - f_4 - \cdots$ and since $v_2\cdot w_1=0$ we have either 
$v_2 = f_3-f_4-f_5-\cdots$ or $v_2 = f_3-f_4-f_6-\cdots$. The first case is not possible because it would 
imply $v_2\cdot w_2=2$, therefore the second case occurs and $v_2\cdot w_2=1$, which implies that 
$v_2$ belongs to the same connected component as $w_2$, hence $v_2 = f_3-f_4-f_6$. The third vector $u_2$ 
of the connected component of $S$ containing $v_2$ and $w_2$ must share a vector with both $v_2$ 
and $w_2$. Since $f_1,\ldots, f_4$ have already been used three times, this forces $u_2 = f_6-f_5$. 
Therefore, $S=C_1\cup C_2$ where $C_2 = \{u_2, v_2, w_2\}$ is a connected component of the stated form. 
\end{proof} 

\medskip
\noindent\underline{Suppose $t\geq 3$.} 
By the considerations following Remark~\ref{r:-2-contr}, it is easy to see that $\Z$  
consists of some number $n$ of connected components $C_1,\ldots, C_n$, where each component $C_i$ consists of 
$t$ vectors of square $\leq -3$.

\begin{lemma}\label{l:zeta}
Each $z\in\Z$ satisfies $z\cdot z=-3$.
\end{lemma} 

\begin{proof}
For each $z\in\Z$ we have $z=-\sum_{f\in\F} (z\cdot f) f$. By Lemma~\ref{l:pcsinj}(2) we have 
\[
3|\Z| \leq -\sum_{z\in\Z} z\cdot z = \sum_{z\in\Z} z\cdot \sum_{f\in\F} (z\cdot f) f = 
\sum_{f\in\F} \sum_{z\in\Z} (z\cdot f)^2 = 3 |\F| = 3|\Z|.
\] 
Therefore $3|\Z| = -\sum_{z\in\Z} z\cdot z$. Since $z\cdot z\leq -3$ for each $z\in\Z$, we must have 
$z\cdot z = -3$ for every $z\in\Z$. 
\end{proof}

Observe that, in view of Lemmas~\ref{l:pcsinj} and~\ref{l:zeta}, for each $u\in\Z$ there are distinct 
elements $f_1,f_2,f_3\in\F$ such that $u=f_1-f_2-f_3$. In particular, $\F_u = \{f_1,f_2,f_3\}$.

\begin{lemma}\label{l:intersectingelmsofZ}
Let $u,v\in\Z$ with $u\cdot v=1$. Then, up to swapping $u$ and $v$, one of the following holds:
\begin{enumerate}
\item
$|\F_u\cap\F_v|=1$ and there are five distinct 
elements $f_1,\ldots,f_5\in\F$ such that $u=f_1-f_2-f_3$ and $v=f_3-f_4-f_5$;
\item
$|\F_u\cap\F_v|=3$ and there are three distinct elements $f_1,f_2,f_3\in\F$ such that 
$u=f_1-f_2-f_3$, $v=f_3-f_1-f_2$.
\end{enumerate}
Moreover, if $u$ and $v$ belong to a connected component $C\subset\F$ with 
cardinality $|C|>3$ then Case (1) holds.   
\end{lemma} 

\begin{proof} 
Since $u\cdot f, v\cdot f\in\{-1,0,1\}$ for each $f\in\F$, $u\cdot v=1$ implies that 
$|\F_u\cap\F_v|$ is either $1$ or $3$. In the latter case, up to swapping $u$ and $v$ we 
necessarily have $u=f_1-f_2-f_3$ and $v=f_3-f_1-f_2$ for some distinct $f_1,f_2,f_3\in\F$, 
hence (2) holds. In the first case, up to swapping $u$ and $v$ we must have $\F_u\cap\F_v=\{f_3\}$,  
$u=f_1-f_2-f_3$ and $v=f_3-f_4-f_5$ for some distinct $f_1,\ldots,f_5\in\F$. Therefore, (1) holds. 
If $u$ and $v$ belong to a connected component $C$ with $|C|>3$, there exists an element $w\in C$ such that 
$w\cdot u=1$ and $w\cdot v=0$. Then, since $|\F_w\cap\F_u|$ is odd and $|\F_w\cap\F_v|$ is even, 
we must have $\F_u\neq\F_v$, and only Case (1) can occur.  
\end{proof} 

\begin{lemma}\label{l:ortogonalelmsofZ}
Let $u,v\in\Z$ with $u\cdot v=0$ and $\F_u\cap\F_v\neq\emptyset$. Then, up to swapping $u$ and $v$ 
the following hold: 
\begin{enumerate}
\item
there are four distinct elements $f_1,f_2,f_3,f_4\in\F$ such that $u=f_1-f_2-f_3$ and $v=f_2-f_3-f_4$;
\item 
let $C_u$ and $C_v$ be the connected components of $\Z$ containing $u$ and $v$, and suppose $C_u\neq C_v$.
If $|C_u|=3$, then $|C_v|=3$ and one of the following holds: 
\begin{itemize} 
\item[(a)]
there exist $f_1,\ldots, f_6\in\F$ such that 
\[
C_u=\{f_1-f_2-f_3, f_3-f_4-f_5, f_5-f_6-f_1\},\quad
C_v=\{f_2-f_3-f_4, f_4-f_5-f_6, f_6-f_1-f_2\};
\]
\item[(b)]
there exist $f_1,\ldots, f_8\in\F$ such that 
\[
C_u=\{f_1-f_2-f_3, f_5-f_6-f_1, f_6-f_5-f_1\},\quad
C_v=\{f_4-f_7-f_8, f_2-f_3-f_4, f_3-f_2-f_4\}.
\]
\end{itemize} 
\end{enumerate}
\end{lemma}

\begin{proof} 
(1) Clearly $\F_u\cap\F_v$ contains two elements, say $f_2$ and $f_3$, one of which, say $f_2$, must 
satisfy $(f_2\cdot u, f_2\cdot v)\in\{(-1,1), (1,-1)\}$. Up to swapping $u$ and $v$ we may assume 
that $f_2\cdot u = 1$ and $f_2\cdot v=-1$. Then, $u = f_1 - f_2 - f_3$ for 
some $f_1, f_3\in\F$, and necessarily $v = f_2 - f_3 - f_4$ for some $f_4\in\F$, with $f_1, f_2, f_3$ and $f_4$ 
pairwise distinct. 

(2) Let $u',u''\in C_u$ be the two elements adjacent to $u$, i.e.~such that $u'\cdot u=u''\cdot u=1$. Analogously, 
let $v',v''\in C_v$ be the two elements adjacent to $v$. By Lemma~\ref{l:intersectingelmsofZ} we have 
either $|\F_{u'}\cap\F_u|=1$ or $|\F_{u'}\cap\F_u|=3$. In the latter case we would have 
$u'\in\{f_3-f_1-f_2, f_2-f_1-f_3\}$, which would be incompatible with $u'\cdot v=0$, 
therefore the first case occurs. The only possibilities are $\F_{u'}\cap\F_u = \{f_1\}$ and $\F_{u'}\cap\F_u=\{f_3\}$, 
which correspond, respectively, to $u'=f_5-f_6-f_1$ and $u'=f_3-f_4-f_5$, for some $f_5, f_6\in\F$. 
If the first possibility is realized, then 
it is easy to check that either $u'' = f_3-f_4-f_5$ or $u'' = f_6-f_5-f_1$. 
We are going to analyze the various cases. 

Suppose first that $u'=f_5-f_6-f_1$ and $u'' = f_3-f_4-f_5$. by Lemma~\ref{l:pcsinj}(2) there must be some 
vector $w\not = u$ such that $w\cdot f_2 = 1$. Since $w\cdot u=0$, we must have $w\cdot f_1=1$, and this forces $w=f_6-f_1-f_2$. 
Therefore $w$ is either $v'$ or $v''$, say $v''$. Then, it is easy to check that $|\F_v\cap\F_{v'}|=3$ is 
impossible, therefore $|\F_v\cap\F_{v'}|=1$, which forces $v' = f_4-f_5-f_6$. Thus, $v'$ exhausts both $C_v$ and the 
irreducible component, and (2)(a) holds. 

Now suppose that $u'=f_5-f_6-f_1$ and $u'' = f_6-f_5-f_1$. Then, by Lemma~\ref{l:pcsinj}(2) there 
is some $w\in\Z$ such that $w\cdot f_2=1$. We must have $w\cdot u=0$, 
which implies $w\cdot f_1=1$ or $w\cdot f_3=-1$. But $w\cdot f_1=1$ is incompatible 
with $w\cdot u'=w\cdot u''=0$, therefore we must have $w\cdot f_3=-1$. Then, $w\cdot v>0$ and 
therefore $w = f_3-f_2-f_4\in\{v',v''\}$, so we may assume $w=v''$. If we now consider the only vector $w'\in\Z$ which,  
by Lemma~\ref{l:pcsinj}(2), satisfies $w'\cdot f_4=-1$ we can easily conclude that $w'=f_4-f_7-f_8=v'$ for some $f_7, f_8\in\F$, 
and $v'\cdot v''=1$. Therefore $|C_v=3|$ and $C_u$, $C_v$ are as in (2)(b). 

There remains to examine the possibility $u' = f_3-f_4-f_5$. 
In this case, it is easy to check that $|\F_{u''}\cap\F_u|=3$ is impossible, and $|\F_{u''}\cap\F_u|=1$ 
forces $u''=f_5-f_6-f_1$ for some $f_6\in\F$. As before, this implies $\{v',v''\}=\{f_4-f_5-f_6,f_6-f_1-f_2\}$. 
Therefore $|C_v|=3$ and (2)(a) holds.
\end{proof} 

\begin{prop}\label{p:twoirrcompZ}
Suppose that $u,v\in\Z$ belong to distinct connected components $C_u$, respectively  
$C_v$, and $\F_u\cap\F_v\neq\emptyset$. 
If $|C_u|, |C_v|>3$ then $|C_u|=|C_v|$ and, up to swapping $u$ and $v$ there are distinct elements 
$f_1,\ldots, f_{2|C_u|}\in\F$ such that
\[
C_u = \{ f_{2i-1} - f_{2i} - f_{2i+1}\},\quad C_v = \{f_{2i} - f_{2i+1} - f_{2i+2}\}
\quad\text{for $i=1,\ldots, |C_u|$,}
\]
where $f_{2|C_u|+1} = f_1$ and $f_{2|C_u|+2}=f_2$.
\end{prop} 

\begin{proof} 
According to Lemma~\ref{l:ortogonalelmsofZ}, up to swapping 
$u$ and $v$ we can write $u=f_1-f_2-f_3$ and $v=f_2-f_3-f_4$ for some distinct $f_1,f_2,f_3,f_4\in\F$. 
Let $u',u''\in C_u$ be the two elements adjacent to $u$, i.e.~such that $u'\cdot u=u''\cdot u=1$. Analogously, 
let $v',v''\in C_v$ be the two elements adjacent to $v$.
By Lemma~\ref{l:intersectingelmsofZ} we have 
\[
|\F_{u'}\cap\F_u|=|\F_{u''}\cap\F_u|=|\F_{v'}\cap\F_v|=|\F_{v''}\cap\F_v|=1. 
\]
Moreover, we claim that either $u'\cdot f_1=0$ or $u''\cdot f_1=0$. To prove the claim, suppose by contradiction that 
$u'\cdot f_1\neq 0$ and $u''\cdot f_1\neq 0$. Then, $\F_{u'}\cap\F_u = \F_{u''}\cap\F_u=\{f_1\}$. 
By Lemma~\ref{l:pcsinj}(1) and the surjectivity of the map 
of Lemma~\ref{l:coefficients}, there exists $w\in\Z$ with 
$w\cdot f_2=-1$. This implies, in particular, that $w\not\in\{u',u''\}$, and therefore $w\cdot u=0$. 
By Lemma~\ref{l:ortogonalelmsofZ} we have $w=f_2-f_3-f_5$ 
for some $f_5\in\Z$. But then $w\cdot v<0$, which is impossible. 
Therefore, the claim is proved and without loss of generality we may assume $u'\cdot f_1=0$. Since $u'\cdot u=1$, 
$u'\cdot v=0$ and we already have $v\cdot f_2=-1$, we must have $u'=f_3-f_4-f_5$ for some $f_5\in\F$. If we apply the same argument with the pair $(v,u')$ in place of $(u,v)$ we see that we may assume without loss of generality 
that $v'=f_4-f_5-f_6$ for some $f_6\in\Z$. Now let $u^1:=u$, $u^2:=u'$, $v^1:=v$ and $v^2:=v'$. The same 
argument applied to $(v^1,u^1)$ yields an element $u^3=f_5-f_6-f_7$, and so on. We end up with two sequences 
of the form
\[
u^i = f_{2i-1} - f_{2i} - f_{2i+1}\in C_u,\quad v^i = f_{2i} - f_{2i+1} - f_{2i+2}\in C_v,\ i=1,2,\ldots
\]
such that $u^i\cdot u^{i+1}=v^i\cdot v^{i+1} = 1$ for every $i$. Suppose without loss of generality 
$|C_u|\leq |C_v|$. It is easy to check that, as long as $i\leq |C_u|-1$ the $f_i's$ are all distinct. On the other 
hand, $u^{|C_u|}\cdot u^1=1$ implies $u^{|C_u|} = f_{2|C_u|-1}-f_{2|C_u|} - f_1$ and $u^{|C_u|+1}=u^1$, 
while $v^{|C_u|}\cdot u^1=0$ implies $v^{|C_u|} = f_{2|C_u|} - f_1 - f_2$ and $v^{|C_u|+1}=v^1$. Therefore 
$|C_u|=|C_v|$ and the statement holds.
\end{proof}

\begin{prop}\label{p:oneirredcompZ}
Let $C\subset\Z$ be a connected component of $\Z$ such that, for each $u\in C$ and $v\in\Z$, 
$\F_u\cap\F_v\not=\emptyset$ implies $v\in C$. Then, $|C|$ is odd and there exist 
$f_i\in\F$, $i\in\bZ/(2m+1)\bZ$, such that  
\[
C = \{f_{2i-1}-f_{2i}-f_{2i+1}\ |\ i\in\bZ/(2m+1)\bZ\}. 
\]
\end{prop} 

\begin{proof} 
Without loss of generality we may assume $C=\Z$. 
We prove the statement by induction on $|\Z|\geq 3$, treating separately the two 
cases $|\Z|$ odd and $|\Z|$ even. For the basis of the induction in the odd case, we 
observe that if $|\Z|=3$ it is easy to check that  there exist $f_1,f_2,f_3\in\F$ such that
\[
\Z = \{f_1-f_2-f_3, f_3-f_1-f_2, f_2-f_3-f_1\}, 
\]
therefore the statement holds. In the even case the basis is the case $|\Z|=4$. Then, $|\F|=4$ as well, which 
is impossible by Lemma~\ref{l:intersectingelmsofZ} (Case~$(1)$ applies because $|\Z|>3$). Therefore,  
there is no set $\Z$ satisfying the assumptions of the proposition with $|\Z|=4$. In the proof we will show that 
the same conclusion holds if $|\Z|$ is even. 

Now suppose that $|\Z|\geq 5$, and let $u,v\in \Z$ with $u\cdot v=1$. 
By Lemma~\ref{l:intersectingelmsofZ}, up to renaming $u$ and $v$ there exist five distinct elements 
$f_1,\ldots, f_5\in\F$ such that $u=f_1-f_2-f_3$ and $v=f_3-f_4-f_5$. Let $w$ be the only element of $\Z$ such 
that $w\neq u$ and $w\cdot f_3=1$. By Lemma~\ref{l:intersectingelmsofZ} we must have $w\cdot u=0$. 
We claim that $w\cdot v=0$ as well. Suppose by contradiction that $w\cdot v=1$. Then, let $w'$ and $w''$ be the unique 
elements of $\Z$ such that $w'\cdot f_4=-1$ and $w''\cdot f_5=-1$. Since necessarily $w'\cdot v=w''\cdot v=0$ 
and $f_3$ already hits $u, v$ and $w$, we must have $w'\cdot f_5=w''\cdot f_4=1$. But this would imply 
$w'\cdot w''>0$ and $|\F_{w'}\cap\F_{w''}|>1$, contradicting Lemma~\ref{l:intersectingelmsofZ}. Therefore 
the claim is proved and $w\cdot v=0$. Up to swapping $f_4$ and $f_5$ 
we have $w=f_2-f_3-f_4$. Now let $w' = f_2 - f_4$, and consider the new set 
\[
\Z' = \Z\setminus\{u,v,w\} \cup \{u+v,w'\} \subset\Z.
\]
By construction, $\Z'$ has clearly cardinality $|\Z|-1$ and is contained in the span of $\F'=\F\setminus\{f_3\}$. Thus, 
it is a positive, circular subset in $\bZ^{|\Z'|}$ admitting a canonical adapted basis, and 
the map of Lemma~\ref{l:coefficients} is injective. 
Moreover, the Wu element $W'$ of $\Z'$ clearly satisfies $W'\cdot W'=-|\Z'|$. 
Hence, the assumptions of Lemma~\ref{l:pcsinj} are satisfied for $\Z'$. 
We can therefore apply a (-2)--contraction using the 
vector $w'$. Since $u+v=f_1-f_2-f_4-f_5$, the (-2)--contraction gives a new positive circular subset $\Z''$ of 
cardinality $|\Z''|=|\Z'|-1=|\Z|-2\geq 3$ contained in the span of $\F''=\F\setminus\{f_2,f_3\}$. By construction, all 
the elements of $\Z''$ have square $-3$, and $\Z''$ has a single connected component. Moreover, it is 
easy to check that the assumptions of Lemma~\ref{l:pcsinj} are still satisfied by $\Z''$. 
If $|\Z|$ and $|\Z''|$ were even, we could repeat the same construction -- several times if necessary --  
to obtain a set $\widetilde\Z$ with $|\widetilde\Z|=4$, satisfying the assumptions of the proposition, which 
we have already shown to be impossible. Therefore $|\Z|$ and $|\Z''|$ must be odd. 
By the inductive assumption there exist $f_i\in\F''$, $i\in\bZ/(2m-1)\bZ$, such that
\[
\Z'' = \{f_{2i-1}-f_{2i}-f_{2i+1}\ |\ i\in\bZ/(2m-1)\bZ\}. 
\]
Observe that $\Z$ is obtained from $\Z''$ by replacing, for some $j\in\bZ/(2m-1)\bZ$, the subset
\[
\{f_{2j-1}-f_{2j}-f_{2j+1}, f_{2j+1}-f_{2j+2}-f_{2j+3}\}
\]
with
\[
\{f_{2j-1}-f_{2j}-f_{2m+3}, f_{2m+2}-f_{2m+3}-f_{2j+1}, f_{2j+1}-f_{2j+2}-f_{2j+3}\}
\]
and the subset
\[
\{f_{2j}-f_{2j+1}-f_{2j+2}\}\quad\text{with}\quad \{f_{2j}-f_{2m+2}-f_{2m+3}, f_{2m+3}-f_{2j+1}-f_{2j+2}\}.
\]
It is now a simple matter to deduce the statement for $\Z$. 
\end{proof} 

\subsection*{Conclusions}
We now have enough information about positive circular subsets of $\bZ^N$, so we can draw 
the conclusions we are interested in. The following theorem will be used in the next section, 
together with Theorem~\ref{t:semipos}, to prove Theorem~\ref{t:main}.

\begin{thm}\label{t:pcs}
Let $\Ga$ be a weighted graph as in Figure~\ref{f:graph} with $d=0$, $t\geq 2$, $x_i, y_i\geq 1$ 
and $\sum_i y_i = \sum_i x_i\geq 3$. 
Let $N=k\sum_i y_i$ for some $k\geq 1$, and suppose that there is an isometric embedding 
$\La^k_\Ga\subset\bZ^N$ of finite and odd index. Then, 
the string of negative integers 
\begin{equation}\label{e:string}
S:=(-2-x_1,\overbrace{-2,\ldots,-2}^{y_1-1},\ldots,-2-x_t,\overbrace{-2,\ldots,-2}^{y_t-1})
\end{equation}
satisfies one of the following:
\begin{enumerate}
\item
$S$ is obtained, up to circular symmetry, from $(-2,-2,-5)$ by iterated (-2)--expansions 
(in the sense of Definition~\ref{d:-2-exp});
\item
there exist a regular polygon $P$ with $t$ vertices and a symmetry $\varphi\co P\to P$ such that
$(X,Y) = (Y,X')^\varphi$, where $(X,Y)$ is the labelling of $P$ encoded by $(x_1,y_1,\ldots, x_t,y_t)$ 
and $(Y,X')$ is the labelling encoded by $(y_1,x_2, y_2,x_3,\ldots, y_t,x_1)$.
\end{enumerate}
\end{thm}

\begin{proof} 
We proceed as in Theorem~\ref{t:semipos}. The lattice $\La^k_\Ga$ has a natural basis, whose elements are in 1--1 
correspondence with the vertices of the disjoint union of $k$ copies of $\Ga$, and 
intersect as prescribed by edges and weights. Since $d=0$, the image of such a basis is a positive circular 
subset $\V\subset\bZ^N$. The Wu element $W=\sum_{v\in\V} v$ satisfies 
\[
W\cdot W = k\left(\sum_i (-2-x_i) -2\sum_i (y_i-1) + 2\sum_i y_i\right) = 
-k\sum_i y_i = -N
\]
and since $\La_\V=\La^k_\Ga$, the embedding $\La_\V\subset\bZ^N$ has finite and odd index. 
By Lemma~\ref{l:wuelement}, there is a canonical basis $\E\subset\bZ^N$ adapted to $\V$. 
Moreover, each connected component $D\subset\V$ satisfies 
\[
\sum_{v\in D} (v\cdot v+2) = -\sum_i x_i = -\sum_i y_i = -|D|.
\]
If the map of Lemma~\ref{l:coefficients} is not injective then the assumptions of Proposition~\ref{p:pcsnotinj} 
are satisfied. Applying the proposition, Case~$(1)$ of the statement follows immediately. 

If the map of Lemma~\ref{l:coefficients} is injective, the considerations following Remark~\ref{r:-2-contr} 
show that applying a finite number of (-2)--contractions (in the sense of Definition~\ref{d:-2-contr}) 
to the subset $\V$ one obtains a circular subset $\Z\subset\bZ^{|\Z|}$ satisfying all the assumptions of 
Lemma~\ref{l:pcsinj} and having the property that each $(-2)$--vector of $\Z$ belongs to 
a connected component with exactly three elements. 
In particular, there is a canonical basis $\F\subset\bZ^{|\Z|}$ adapted to $\Z$. 
Clearly, every irreducible component of $\Z$ is a union of 
connected components. Since the string $S$ given by~\eqref{e:string} is associated 
to any connected component of $\V$, we can assume 
without loss of generality that $\Z$ is irreducible. 

When $t=2$ the subset $\Z$ satisfies the conclusions of Proposition~\ref{p:Zcaset=2}, therefore it is 
the union of two connected components $C_1 = \{u_1,v_1,w_1\}$ and $C_2 = \{u_2,v_2,w_2\}$, 
where $u_1 = f_1-f_2$, $v_1 = f_2-f_3-f_4$, $w_1 = f_4-f_5-f_6-f_1$ and 
$u_2 = f_6-f_5$, $v_2 = f_3-f_4-f_6$ and  $w_2 = f_5-f_1-f_2-f_3$ for some $f_1,\ldots, f_6\in\F$. 
Now we want to recover the self--intersections of the elements of $\V$ using the fact that $\Z$ is 
obtained from $\V$ by $(-2)$--contractions. Of course, $\V$ is obtained from $\Z$ 
by $(-2)$--expansions. Suppose that, for $i=1,2$, $\V$ is obtained by applying $a_i$ 
$(-2)$--expansions between $u_i$ and $v_i$, $b_i$ $(-2)$--expansions between $v_i$  and $w_i$, 
and $c_i$ $(-2)$--expansions between $w_i$ and $u_i$. Then, it is easy to check that $v_i\in\Z$ 
becomes a vector of $\V$ of self--intersection $-3-b_{3-i}$ and $w_i$ a vector of 
self--intersection $-4-a_{3-i}-c_{3-i}$, for $i=1,2$. Therefore, up to a symmetry, the weights of the weighted 
graphs associated to the two connected components of $\V$ are 
\[
(-4-a_2-c_2,\overbrace{-2,\ldots, -2}^{a_1+c_1+1}, -3-b_2,\overbrace{-2,\ldots,-2}^{b_1})
\]
and
\[
(-4-a_1-c_1,\overbrace{-2,\ldots, -2}^{a_2+c_2+1}, -3-b_1,\overbrace{-2,\ldots,-2}^{b_2}).
\]
Since the two components of $\V$ had the same self--intersections, in terms of  
the parameters $x_i, y_i$ this easily translates into the condition that either (i) $x_1=y_1$ and $x_2=y_2$ 
or (ii) $x_1=y_2$ and $x_2=y_1$. In both cases, there exist a bigon $P$ and a symmetry $\varphi\co P\to P$ 
such that $(X,Y) = (Y,X')^\varphi$, where $(X,Y)$ is the labelling of $P$ encoded by $(x_1,y_1, x_2, y_2)$ 
and $(Y,X')$ is the labelling encoded by $(y_1,x_2, y_2,x_1)$. In fact, in the first case we can take the 
reflection of $P$ which fixes the vertices and switches the edges, in the second case the reflection 
which switches the vertices and fixes the edges. Therefore Case~(2) holds. 

When $t\geq 3$ we first assume that $\Z$ has more than one connected component. 
Then, by Lemma~\ref{l:ortogonalelmsofZ} and Proposition~\ref{p:twoirrcompZ} 
all the connected components of $\Z$ have the same cardinality. 
As in the case $t=2$, we want to recover the self--intersections of the elements of $\V$ from the structure of $\Z$.
Let $C\subset\Z$ be a connected component. 
If $|C|>3$ then Proposition~\ref{p:twoirrcompZ} applies, and we shall now refer to the notation and 
terminology of that proposition.  
Starting from $\Z$, if we do, say, $x_i$ (-2)--expansions between $f_{2i-1}-f_{2i}-f_{2i+1}$ 
and $f_{2i+1}-f_{2i+2}-f_{2i+3}$, the weight of the vertex corresponding to 
$f_{2i}-f_{2i+1}-f_{2i+2}$ becomes $-3-x_i$. Similarly, doing $y_i$ 
(-2)--expansions between $f_{2i}-f_{2i+1}-f_{2i+2}$ and $f_{2i+2}-f_{2i+3}-f_{2i+4}$ makes 
the vertex corresponding to $f_{2i-1}-f_{2i}-f_{2i+1}$ acquire 
weight $-3-y_i$. This shows that the two connected components $C_u$ and $C_v$ of 
Proposition~\ref{p:twoirrcompZ} come from two connected components 
$\widetilde C_u, \widetilde C_v\subset\V$ having self--intersections given, 
up to a symmetry of the corresponding weighted graphs, by 
\[
(\overbrace{2,\ldots,2}^{x_1},-3-y_1,\overbrace{2,\ldots,2}^{x_2},
\ldots,\overbrace{2,\ldots,2}^{x_t},-3-y_t)\quad\text{and}\quad 
(-3-x_1,\overbrace{2,\ldots,2}^{y_1},-3-x_2,
\ldots,-3-x_t,\overbrace{2,\ldots,2}^{y_t}).
\]
Since the weighted graphs associated to $\widetilde C_u$ and $\widetilde C_v$ are isomorphic, 
it follows easily that there is a regular $t$--gon $P$ and a symmetry $\varphi\co P\to P$ such that
$(X,Y) = (Y,X')^\varphi$, where $(X,Y)$ is the labelling of $P$ encoded by $(x_1,y_1,\ldots, x_t,y_t)$ 
and $(Y,X')$ is the labelling encoded by $(y_1,x_2, y_2,x_3,\ldots, y_t,x_1)$. 
Therefore Case~(2) holds. If $|C|=3$ then Lemma~\ref{l:ortogonalelmsofZ}(2) applies. Conclusion (a) in the statement 
of the lemma is perfectly analogous to the statement of Proposition~\ref{p:twoirrcompZ}. 
Therefore, if Case~(a) in the lemma holds, Case~(2) of the statement holds.  
If the conclusion of Lemma~\ref{l:ortogonalelmsofZ}(2)(b) holds, an analysis similar to the case 
$|C|>3$ shows that the connected component $C_u$ corresponds
to a connected component $\widetilde C_u\subset\V$ with string of self--intersections given, up to a symmetry 
of the corresponding weighted graph, by 
\[
(-3-x_1,\overbrace{2,\ldots,2}^{x_3},-3-x_2,\overbrace{2,\ldots,2}^{x_1}, -3-x_3,\overbrace{2,\ldots,2}^{x_2}).
\]
In other words, up to symmetry, the string is of the type given by Equation~\eqref{e:string} 
with $t=3$ and $(y_1,y_2,y_3)=(x_3,x_1,x_2)$. In other words, there is an equilateral triangle $P$ and 
a $2\pi/3$--rotation $\varphi\co P\to P$ such that $(x_1,y_1,x_2,y_2,x_3,y_3) = (y_1,x_2,y_2,x_3,y_3,x_1)^\varphi$.  
Therefore Case~(2) holds again.

Now we assume that $\Z$ is connected. Then, Proposition~\ref{p:oneirredcompZ} applies. This case is quite 
similar to the case of Proposition~\ref{p:twoirrcompZ}. For each index $i$, we can only do $x_i$ (-2)--expansions 
between $f_{2i-1}-f_{2i}-f_{2i+1}$ and $f_{2i+1}-f_{2i+2}-f_{2i+3}$. Then, the vertex corresponding to 
$f_{2i}-f_{2i+1}-f_{2i+2}$ acquires weight $-3-x_i$. This implies that the weights of each 
component of $\V$ are given by Equation~\eqref{e:string}, with $t=2m+1$ and 
$y_i = x_{i+m+1}$ for each $i\in\bZ/(2m+1)\bZ$. Arguing as when $\Z$ is disconnected, 
we see again that Case~(2) holds. This concludes the proof.
\end{proof} 

\section{The proof of Theorem~\ref{t:main}}\label{s:proof}

In this section we use the results of Sections~\ref{s:prelim} and~\ref{s:latan} to prove Theorem~\ref{t:main}. 
By Proposition~\ref{p:prelim}, $K$ is the closure of a 3--braid $\be\in B_3$ of the form 
\begin{equation*}
\be = (\si_1\si_2)^{3d} \si_1^{x_1}\si_2^{-y_1}\cdots\si_1^{x_t}\si_2^{-y_t},\quad
t, x_i, y_i\geq 1, 
\end{equation*}
for some $d\in\{-1,0,1\}$ with $\sum_{i=1}^t (x_i-y_i) = -4d$. 
We  consider separately the two cases $d=\pm 1$ and $d=0$. 

{\em First case: $d=\pm 1$}. If $d=-1$ then $K^m$ is the closure of an analogous 3--braid with $d=1$. 
Therefore, we will assume without loss that $d=1$. 
By Proposition~\ref{p:embed}, for some $k\geq 1$ the orthogonal sum $\La^k_\Ga$  
of $k$ copies of $\La_\Ga$ embeds isometrically, with finite and odd index, 
in the standard negative definite lattice $\bZ^N$, where $N=k\sum_i y_i$.
If $\sum_i x_i = 1$ then $t=1$, $x_1=1$ and $y_1=x_1+4=5$. Therefore $\be = (\si_1\si_2)^3\si_1\si_2^{-5}$, 
which belongs to Family~$(1)$ of Theorem~\ref{t:main}. 
Indeed, $(x_1+2, \overbrace{2,\ldots,2}^{y_1-1}) = (3,2,2,2,2)$ and it is 
easy to check that $(3,1,2,2,1)$ is an iterated blowup of $(0,0)$ and $\be$ is quasi--positive because it is 
conjugate to $\si_2^3\si_1\si_2^{-3}\si_1$, or by~\cite[Theorem~2.3]{Li14}. 
We have $\lk(\be)=2$, and we show below that a knot with vanishing signature which is the closure 
of a quasi--positive 3--braid having exponent sum equal to $2$ is ribbon. Thus, we may assume 
$\sum_i y_i = \sum_i x_i + 4 \geq 6$. Theorem~\ref{t:semipos} implies that $K$ belongs to Family~$(1)$.  
To conclude the proof in this case we need to argue that $\be$ is quasi-positive and $K$ is ribbon. 
As observed in Section~\ref{s:intro}, the fact that $\be$ is quasi-positive follows from~\cite[Theorem~2.3]{Li14}.
Next, we observe that since $K$ has finite concordance order its signature vanishes, therefore by 
Equation~\eqref{e:signature} we have $e(\be) = 2d = 2$, where $e\co B_3\to\bZ$ is the exponent sum 
homomorphism. Finally, the constructions described in~\cite[\S~2]{Ru83} show that the closure of a quasi-positive 
3--braid with exponent sum equal to $2$ bounds an immersed surface $S\subset S^3$ with only ribbon singularities, 
constructed as the disjoint union of three embedded disks and two embedded 1--handles which intersect the disks 
in ribbon singularities. Since $K=\hat\be = \del S$ is a knot, the fact that $\chi(S) = 3 - 2 = 1$ implies that in our case we can find such an $S$ which an immersed ribbon disk. It is well--known that suitably pushing the 
interior of $S$ into the 4--ball we obtain a ribbon disk for $K$.  

{\em Second case: $d=0$}.  
Let $n:=\sum_i x_i=\sum_i y_i$. If $n=1$ we have $t=x_1=y_1=1$, 
therefore $\be=\si_1\si_2^{-1}$, $K=\hat\be$ belongs to Family~(3) and it is the unknot, which is 
amphichiral. If $n=2$ there are two possible cases: either $t=2$ and $x_1=y_1=x_2=y_2=1$, 
or $t=1$ and $x_1=y_1=2$. In the first case $\be=(\si_1\si_2^{-1})^2$, $K$ belongs to Family~(3) 
and it is the figure--eight knot, which is well--known to be amphichiral. In the second case 
$\be =\si_1^2\si_2^{-2}$, and $K=\hat\be$ is a 3--component link. Therefore, 
from now on we assume $n=\sum_i x_i=\sum_i y_i\geq 3$. 

By Proposition~\ref{p:embed}, for some $k\geq 1$ the orthogonal sum $\La^k_\Ga$  
of $k$ copies of $\La_\Ga$ embeds isometrically, with finite and odd index, 
in the standard negative definite lattice $\bZ^N$, where $N=k\sum_i y_i$.
The embedding $\La^k_\Ga\subset\bZ^N$ gives rise to a positive circular subset 
$\V\subset\bZ^N$ such that $v\cdot v\leq -2$ for each $v\in\V$. 
By Theorem~\ref{t:pcs}, the string $S$ of negative integers of Equation~\eqref{e:string} falls under 
Case~$(1)$ or Case~$(2)$ of the statement of that theorem. 
We will now verify that in Case~$(1)$ the knot $K$ belongs to Family~$(2)$ 
of Theorem~\ref{t:main}. We shall argue by induction on the length of the string $S$.
Observe first that $L_a=\widehat\be_a$, where 
\[
\be_a := \si_2\si_1\si_2^{-1} a \si_2 \si_1^{-1}\si_2^{-1} a^{-1} \in B_3. 
\]
The 3--braid $\be_a$ is conjugate to $\be'_a = \si_2^{-2} a \si_2^2 \De^{-1} a^{-1} \De$, 
where $\De = \si_1\si_2\si_1\in B_3$ is the Garside element. To prove the basis of the induction,  
note that the knot corresponding to $(-2,-2,-5)$ is 
the closure of $\si_2^{-3}\si_1^3\in B_3$, and we may also write this braid as  
\[
\be'_{\si_2^{-1}\si_1^2} = \si_2^{-2} (\si_2^{-1}\si_1^2) \si_2^2 \De^{-1} (\si_2^{-1}\si_1^2)^{-1}\De.
\]
This shows that $K$ belongs to Family~$(2)$. To prove the inductive step, 
we first consider the two (-2)--expansions of $(-2,-2,-5)$, which are $(-2,-2,-2,-6)$ and 
$(-2,-3,-5,-2)$. As one can easily check, the knot corresponding to the first expansion is the closure 
of the braid obtained from $\be'_{\si_2^{-1}\si_1^2}$ by inserting $\si_1\si_2^{-1}$ immediately 
before the factor $\si_2^{-2}$, while the knot corresponding to the second expansion is the closure 
of the braid obtained inserting $\si_2^{-1}\si_1$ immediately after the factor $\si_2^{-2}$. 
Observe that, in general,  since $\De\si_1 = \si_2\De$, 
\[
\si_1\si_2^{-1} \be'_a  = 
 \si_1\si_2^{-2} (\si_2^{-1} a) \si_2^2 \De^{-1} a^{-1} \De \sim 
\si_2^{-2} (\si_2^{-1} a) \si_2^2 \De^{-1} (\si_2^{-1} a)^{-1} \De = \be'_{\si_2^{-1} a},  
\]
where $\sim$ denotes conjugation. Similarly, 
\[
\si_2^{-2} (\si_2^{-1}\si_1) a \si_2^2 \De^{-1} a^{-1} \De \sim \si_2^{-2} (\si_1 a) \si_2^2 \De^{-1} (\si_1 a)^{-1}\De 
= \be'_{\si_1 a}.
\]
This shows that the knots corresponding to the two (-2)--expansions of $(-2,-2,-5)$ are of the form 
$\widehat\be_{\si_2^{-2}\si_1^2}$ and $\widehat\be_{\si_1\si_2^{-1}\si_1^2}$. In general, 
suppose that the string $(-m_1,\ldots, -m_k)$ was obtained by a sequence of 
(-2)--expansions from $(-2,-2,-5)$, and that the corresponding knot was the closure of a braid of 
the form $\be'_a$. Arguing as before we see that the knots corresponding to the 
two (-2)--expansions of $(-m_1,\ldots, -m_k)$ are closures of braids obtained from
$\be'_a$ by inserting either $\si_1\si_2^{-1}$ immediately before the factor 
$\si_2^{-2}$ or $\si_2^{-1}\si_1$ immediately after the factor $\si_2^{-2}$. Exactly as 
before, the results are $\be'_{\si_2^{-1} a}$ and $\be'_{\si_1 a}$. This proves the inductive 
step and shows that if Case~$(1)$ of Theorem~\ref{t:pcs} holds, then $K$ belongs to 
Family~$(2)$ of Theorem~\ref{t:main}. The fact that the link $L_a$  
is a symmetric union, hence ribbon, is evident by looking at Figure~\ref{f:symmunion}. 

Before considering the next case, observe that if $t=1$ then $x=y\geq 3$, and $K$ is the closure of $\si_1^x\si_2^{-x}$. 
In this case, the string $S$ given by Equation~\eqref{e:string} is of the form given in Case~$(1)$ of the statement 
of Theorem~\ref{t:pcs}. Therefore, by the above argument $K$ belongs to Family~$(2)$ of Theorem~\ref{t:main}.
From now on we will assume that $t\geq 3$. 

Now suppose that the string $S$ given by Equation~\eqref{e:string} falls under Case~$(2)$ 
of Theorem~\ref{t:pcs}. This immediately implies that $K$ belongs to Family~(3) 
of Theorem~\ref{t:main}, and we are only left to show that $K$ is amphichiral.
We start observing that the condition $(X,Y)=(Y,X')^\varphi$ translates, using the notations 
of Section~\ref{s:intro}, into 
\begin{equation}\label{e:phicondition}
x_i = y_{\varphi_E(i)},\quad y_i = x_{\varphi_V(i+1)},\quad i=1,\ldots, t, 
\end{equation}
where from now on `$i+1$' will mean `$1$' when $i=t$. Moreover, it is easily checked that if $\varphi$ is a rotation then 
\begin{equation}\label{e:phi=rot}
\varphi_V(i+1) = \varphi_V(i)+1\quad \text{and}\quad \varphi_E(i) = \varphi_V(i),
\end{equation}
while if $\varphi$ is a reflection then 
\begin{equation}\label{e:phi=refl}
\varphi_V(i+1),= \varphi_V(i)-1\quad \text{and}\quad \varphi_E(i) = \varphi_V(i)-1.
\end{equation}
By Equation~\eqref{e:phicondition}, the knot $K$ is the closure of 
\[
\prod_{i=1}^t \si_1^{x_i} \si_2^{-y_i} = \prod_{i=1}^t \si_1^{x_i} \si_2^{-x_{\varphi_V(i+1)}},
\]
therefore $K^m$ is the closure of 
\[
\prod_{i=1}^t \si_1^{-x_i} \si_2^{x_{\varphi_V(i+1)}}, 
\]
therefore (applying an isotopy) also the closure of 
\[
\prod_{i=1}^t \si_2^{-x_i} \si_1^{x_{\varphi_V(i+1)}} \sim 
\prod_{i=1}^t \si_1^{x_{\varphi_V(i+1)}} \si_2^{-x_{i+1}} \sim 
\prod_{i=1}^t \si_1^{x_{\varphi_V(i)}} \si_2^{-x_i}  = 
\prod_{i=1}^t \si_1^{x_{\varphi_V(i)}} \si_2^{-y_{\varphi_E(i)}}. 
\]
where `$\sim$' denotes conjugation and to obtain the equality we applied Equation~\eqref{e:phicondition}.
If $\varphi$ is a rotation then $(\varphi_V(1),\ldots, \varphi_V(t)) = (1,2,\ldots, t)$, hence  
by Equation~\eqref{e:phi=rot} we have 
\[
\prod_{i=1}^t \si_1^{x_{\varphi_V(i)}} \si_2^{-y_{\varphi_E(i)}} = 
\prod_{i=1}^t \si_1^{x_{\varphi_V(i)}} \si_2^{-y_{\varphi_V(i)}} \sim 
\prod_{i=1}^t \si_1^{x_i} \si_2^{-y_i},
\] 
therefore $K^m$ is isotopic to $K$. If $\varphi$ is a reflection then 
$(\varphi_V(1),\ldots, \varphi_V(t)) = (t,t-1,\ldots,1)$, hence by Equation~\eqref{e:phi=refl} we have 
\[
\prod_{i=1}^t \si_1^{x_{\varphi_V(i)}} \si_2^{-y_{\varphi_E(i)}} = 
\prod_{i=1}^t \si_1^{x_{\varphi_V(i)}} \si_2^{-y_{\varphi_V(i)-1}} \sim 
\prod_{i=1}^{t-1} \si_1^{x_{t+1-i}} \si_2^{-y_{t-i}}\cdot \si_1^{x_1}\si_2^{-y_t} \sim
\prod_{i=0}^{t-1} \si_2^{-y_{t-i}} \si_1^{x_{t-i}},
\] 
therefore $K^m$ is isotopic to $-K$.  This concludes the proof of Theorem~\ref{t:main}.

\bibliographystyle{amsplain}
\bibliography{3braidtorsion}
\end{document}